\documentclass{amsart}

\usepackage{amssymb, amscd, latexsym, graphicx, psfrag}
\usepackage[all]{xy}

\newtheorem{dummy}{dummy}[section]
\newtheorem{lemma}[dummy]{Lemma}
\newtheorem{theorem}[dummy]{Theorem}
\newtheorem*{untheorem}{Main Theorem}
\newtheorem{conjecture}[dummy]{Conjecture}
\newtheorem{corollary}[dummy]{Corollary}
\newtheorem{proposition}[dummy]{Proposition}
\theoremstyle{definition}
\newtheorem{definition}[dummy]{Definition}
\newtheorem{example}[dummy]{Example}
\newtheorem{remark}[dummy]{Remark}
\newtheorem{remarks}[dummy]{Remarks}
\newtheorem{question}[dummy]{Question}

\newcommand{\bP}{\mathbb{P}}

\newcommand{\bR}{\mathbb{R}}
\newcommand{\bZ}{\mathbb{Z}}

\newcommand{\bfP}{\mathbf{P}}

\newcommand{\cD}{\mathcal{D}}
\newcommand{\cE}{\mathcal{E}}
\newcommand{\cF}{\mathcal{F}}
\newcommand{\cG}{\mathcal{G}}
\newcommand{\cH}{\mathcal{H}}
\newcommand{\cI}{\mathcal{I}}

\newcommand{\cL}{\mathcal{L}}

\newcommand{\cP}{\mathcal{P}}
\newcommand{\cQ}{\mathcal{Q}}
\newcommand{\cO}{\mathcal{O}}
\newcommand{\cS}{\mathcal{S}}
\newcommand{\cT}{\mathcal{T}}

\newcommand{\cZ}{\mathcal{Z}}

\newcommand{\sfR}{\mathsf{R}}

\newcommand{\Hom}{\mathrm{Hom}}
\newcommand{\dghom}{\mathit{hom}}
\newcommand{\uhom}{\underline{\mathit{hom}}}

\newcommand{\Si}{\Sigma}
\newcommand{\bGamma}{\mathbf{\Gamma}}

\newcommand{\LS}{ {\Lambda_\Sigma} }

\newcommand{\uchi}{\underline{\chi}}
\newcommand{\Spec}{\mathrm{Spec}\,}
\newcommand{\Ext}{\mathrm{Ext}}
\newcommand{\Tor}{\mathrm{Tor}}

\newcommand{\Sh}{\mathit{Sh}}
\newcommand{\naive}{\mathit{naive}}
\newcommand{\Tr}{\mathit{Tr}}
\newcommand{\Perf}{\cP\mathrm{erf}}

\newcommand{\Shard}{\mathrm{Shard}}

\newcommand{\finfib}{\mathit{fin}}

\newcommand{\FT}{\cF\cT}
\renewcommand{\SS}{\mathit{SS}}
\newcommand{\ltr}{\langle \Theta \rangle}

\newcommand{\ltrp}{\langle\Theta'\rangle}

\newcommand{\pol}{\mathit{pol}}
\newcommand{\ori}{\mathfrak{or}}
\newcommand{\Gm}{\mathbb{G}_{\mathrm{m}}}

\newcommand{\tLS}{\widetilde{\Lambda}_\Sigma}

\begin{document}
\title
{A categorification of Morelli's theorem}

\author{Bohan Fang}
\address{Bohan Fang, Department of Mathematics, Columbia University,
2990 Broadway, New York, NY 10027}
\email{b-fang@math.columbia.edu}

\author{Chiu-Chu Melissa Liu}
\address{Chiu-Chu Melissa Liu, Department of Mathematics, Columbia University,
2990 Broadway, New York, NY 10027}
\email{ccliu@math.columbia.edu}

\author{David Treumann}
\address{David Treumann, Department of Mathematics, Northwestern University,
2033 Sheridan Road, Evanston, IL  60208}
\email{treumann@math.northwestern.edu}

\author{Eric Zaslow}
\address{Eric Zaslow, Department of Mathematics, Northwestern University,
2033 Sheridan Road, Evanston, IL  60208}
\email{zaslow@math.northwestern.edu}

\maketitle

\begin{abstract}
We prove a theorem relating torus-equivariant coherent sheaves on toric varieties
to polyhedrally-constructible sheaves on a vector space. 
At the level of K-theory, the theorem recovers
Morelli's description of the K-theory of a smooth projective toric variety \cite{M}.
Specifically, let $X$ be a proper toric variety of dimension $n$ and
let $M_\bR = \mathrm{Lie}(T_\bR^\vee)\cong \bR^n$ be the Lie algebra 
of the compact dual (real) torus $T_\bR^\vee\cong U(1)^n$.
Then there is a corresponding conical Lagrangian $\Lambda \subset T^*M_\bR$
and an equivalence of triangulated dg categories
$\Perf_T(X) \cong \Sh_{cc}(M_\bR;\Lambda),$
where
$\Perf_T(X)$ is the triangulated dg category of perfect complexes of 
torus-equivariant coherent sheaves on $X$  and
$\Sh_{cc}(M_\bR;\Lambda)$ is the triangulated dg category of complex of sheaves on $M_\bR$ with 
compactly supported, constructible cohomology whose singular support lies in $\Lambda$.
This equivalence is monoidal---it intertwines the tensor product of coherent
sheaves on $X$ with the convolution product of constructible sheaves on $M_\bR$.
\end{abstract}

\section{Introduction}

In this paper we construct a dictionary between equivariant coherent sheaves on toric varieties and
constructible sheaves on real periodic hyperplane arrangements.

Recall that to a projective $n$-dimensional toric variety $X$ we may associate a polytope $\Delta \subset \bR^n$ whose vertices are lattice points.
%Our correspondence assigns to the hyperplane line bundle the constant sheaf on the interior of $\Delta,$ extended by zero.
Let $\cH_\Delta$ denote the union of all affine hyperplanes in $\bR^n$ that are spanned by lattice points and that are parallel to the faces of $\Delta$.  The following is a special case of our main theorem:

\begin{theorem}
\label{thm:1.1}
Let $X$ be an $n$-dimensional smooth projective toric variety and let $\Delta$ be its moment polytope.  Then there is a full embedding of derived categories
$$\kappa: D^b_T(X) \hookrightarrow D^b_{cc}(\bR^n;\cH_\Delta)$$
where
\begin{itemize}
\item $D^b_T(X)$ denotes the bounded derived category of torus-equivariant coherent sheaves on $X$ and
\item $D^b_{cc}(\bR^n;\cH_\Delta)$ denotes the bounded derived category of sheaves on $\bR^n$ which are compactly supported and constructible (``cc'') with respect to the stratification defined by $\cH_\Delta$.
\end{itemize}
The embedding is monoidal---it intertwines the tensor product on $D^b_T(X)$ with a convolution product on $D^b_{cc}(\bR^n;\cH_\Delta)$.
\end{theorem}
This kind of ``coherent-constructible correspondence'' was first observed by Bondal \cite{Bondal}.
By taking Euler characteristics of stalks, the Grothendieck group of the constructible category is naturally identified with the ring of compactly supported functions $\mathrm{Fun}(\cH_\Delta)$ on $\bR^n$ that are constant along each cell of $\cH_\Delta$.  The multiplication in this ring is ``convolution with respect to Euler characteristic measure.''  Theorem \ref{thm:1.1} therefore identifies the equivariant $K$-theory of $X$ with a subring of this ring of functions.  This is the beautiful correspondence of Morelli \cite{M}, which carries the class of an ample line bundle on $X$ to the indicator function of its associated moment polytope.  We regard the results of this paper as ``categorifications'' of those of Morelli.

Morelli showed that the image of the inclusion $K_T(X) \hookrightarrow \mathrm{Fun}(\cH_\Delta)$ can be described by conditions that are local on $\bR^n$.  Our categorification of this result makes use of the microlocal perspective of sheaves developed by Kashiwara and Schapira \cite{KS}.  The microlocal theory of sheaves associates to each constructible sheaf $F$ on $\bR^n$ a conical Lagrangian subvariety $\SS(F) \subset T^* \bR^n$, called its ``singular support.''  A sheaf $F$ is constructible with respect to a stratification $\cS$ if and only if $\SS(F)$ belongs to the ``conormal variety'' $\Lambda(\cS)$ of $\cS$.

We will associate to any toric variety a conical Lagrangian $\Lambda \subset \Lambda(\cH_\Delta)$ which determines image of the embedding $\kappa$ completely.  The ``smooth and projective'' hypotheses can be removed:

\begin{untheorem}
If $X$ is a proper toric variety, there is a corresponding conical Lagrangian $\Lambda \subset T^* \bR^n$ and an equivalence of derived categories (or rather, triangulated dg categories)
$$\Perf_T(X) \cong \Sh_{cc}(\bR^n;\Lambda)$$
where
\begin{itemize}
\item $\Perf_T(X)$ is the triangulated dg category of \emph{perfect} complexes of torus-equivariant coherent sheaves.  I.e., the category of those complexes that locally admit a bounded resolution by torus-equivariant vector bundles.  %When $X$ is smooth, this coincides with the usual derived category of coherent sheaves.

\item  $\Sh_{cc}(\bR^n;\Lambda)$ is the triangulated dg category of sheaves on $\bR^n$ which are compactly supported, constructible, and whose singular support lies in $\Lambda$.
\end{itemize}
Moreover, this equivalence is monoidal---it intertwines the tensor product on $X$ with the convolution product on $\bR^n$.
\end{untheorem}

\begin{remark}
When $X$ is smooth, $\Perf_T(X)$ coincides with (or rather, its homotopy category coincides with) the usual bounded derived category of coherent sheaves.
\end{remark}

\subsection{HMS motivation}

The correspondence described in this paper has some relevance to the homological mirror symmetry (HMS) conjecture of Kontsevich.  Kontsevich's conjecture is a mathematical description of the duality in physics
between $\mathcal N = (2,2)$ supersymmetric
quantum field theories with boundary.  Kontsevich formulated the mirror duality as an equivalence of two categories, with objects representing boundary conditions.  On one side of the mirror equivalence is the category of coherent sheaves, capturing the complex geometry.  On the other side is a symplectic, Fukaya-type category measuring the quantum intersection theory of Lagrangian submanifolds.  In the setting where the complex manifold is a projective toric variety, the correspondence in this paper can be used to prove a slightly nonstandard version of HMS.  We briefly explain here.  For more details, see \cite{FLTZ}.

In \cite{NZ,N} a ``microlocalization'' theorem is proved, showing the dg category of constructible sheaves on a real analytic manifold is quasi-equivalent to the unwrapped Fukaya category of Lagrangian submanifolds of the cotangent bundle.  If one fixes a subcategory of constructible sheaves by specifying that their microsupports lie in a conical Lagrangian $\Lambda$ in the cotangent bundle, then the corresponding Fukaya subcategory is similarly fixed by requiring Lagrangians to be asymptotic to $\Lambda$.  By microlocalization, then, the theorem in this paper proves that the category of equivariant coherent sheaves on a toric variety is quasi-equivalent to a Fukaya category in the cotangent bundle of a vector space.  This is a (equivariant) version of HMS for toric varieties, but not the usual one involving Lagrangian thimbles in a Lefschetz pencil defined by the superpotential of the Hori-Vafa mirror construction.

The mirror functor is conjectured to involve dualizing Lagrangian torus fibrations in symplectic manifolds \cite{SYZ}.  This ``T-duality'' procedure can be used to map holomorphic line bundles to Lagrangians.  The coherent-constructible correspondence defined in this paper is compatible with T-duality, in the sense that the constructible dual of an ample holomorphic line bundle gives, by microlocalization, a Lagrangian submanifold {\em equivalent} to the T-dual Lagrangian.

\subsection{The proof of the main theorem}

In this section we give an outline of the proofs of our main results.  Let us first indicate how the functor $\kappa:\Perf_T(X) \to \Sh_{cc}(\bR^n)$ is constructed.  A typical strategy for defining a functor between triangulated categories is to find some objects of the source category that generate all objects in some sense, to define $\kappa$ on these objects and maps between them, and then to conclude by a formal argument that $\kappa$ is defined on the whole category.  It is difficult to find a convenient collection of such objects in $\Perf_T(X)$.  For instance it is still not understood when $\Perf(X)$ has a ``full strong exceptional collection'' consisting of line bundles---this is the subject of conjectures of King, Costa, and Miro-Roig \cite{CM}.  We get around this difficulty by passing to larger categories.  We allow on the coherent side some sheaves that are not coherent at all, but only quasicoherent.  Their constructible counterparts do not have compact support.

We find in this larger category an (infinite) strongly exceptional collection $\{\Theta'(\sigma,\chi)\}$ that generates a subcategory of $\cQ_T(X)$ (a dg version of the derived category of quasicoherent equivariant sheaves) containing $\Perf_T(X)$.  The existence of a strongly exceptional collection $\{\Theta(\sigma,\chi)\}$ in $\Sh_c(\bR^n)$ with the same Ext pattern implies the existence of an equivalence between $\ltrp$ and $\ltr$---the subcategories generated by these exceptional collections.  In particular we get a full embedding of $\Perf_T(X)$ into $\Sh_c(\bR^n)$.

It is easy to describe the objects $\Theta'(\sigma,\chi)$ and $\Theta(\sigma,\chi)$.  The former are quasicoherent sheaves of the form $j_* \cO$ endowed with an equivariant structure, where $j$ is the inclusion of an affine open set.  The latter are, up to a shift, extended by zero from a constant sheaf on a translate of a conical, polyhedral open set.  

The construction of the functor $\kappa$ is very easy---all of this is done in Section \ref{sec:three}.  The characterization of the image requires a much more complicated argument.  It breaks up into two parts:
\begin{enumerate}
\item Identify precisely which constructible sheaves can be generated by $\Theta$.  This shows in particular that constructible sheaves with compact support come from quasicoherent sheaves.  
\item  Use a monoidal criterion for perfectness to show that the quasicoherent counterparts of compactly supported sheaves are actually perfect.
\end{enumerate}

The proof of (1) is a two-step induction.  One of the inductions is on the ``height'' of the sheaf, the dimension of the singular support in the cotangent directions.  The other induction is on the number of critical points the sheaf has with respect to linear functions on $M_\bR$---this is essentially the number of height-$h$ pieces of the singular support.  We review some of the microlocal sheaf theory needed to understand these arguments in Section \ref{sec:microlocal}.

The proof of (2) requires us to construct a dictionary between standard constructions on the coherent and constructible sides.  This dictionary is interesting in its own right.  We prove that $\kappa$ intertwines the usual tensor structure on $\Perf_T(X)$ with the convolution structure on $\Sh_{cc}(M_\bR)$, and (as a consequence) that it interwines vector bundle duality with the antipode of Verdier duality, and that pulling back an object of $\Perf_T(X)$ to a toric resolution of singularities of $X$ has no effect on the associated constructible sheaf.  All of these are essentially consequences of a ``functoriality for pullbacks'' result, Theorem \ref{thm:functoriality}.

\emph{Acknowledgments:}  We thank D. Arinkin, B. Bhatt, M. Brion, D. Nadler, P. Schapira, and Z. Yun.  We are especially grateful to Alexei Bondal for his inspiring preprint \cite{Bondal}.  The work of EZ is supported in part by NSF/DMS-0707064.  BF and EZ thank the Pacific Institute for the Mathematical Sciences, where some of this work was performed.

\section{Notation and conventions}
\label{sec:notation}

We work throughout over a commutative noetherian base ring $\sfR$.

\begin{remark}
It is because $\sfR$ plays almost no role in our arguments that we are able to work in this generality.  The reader may instead prefer to assume $\sfR = \bZ$ and deduce the results for other rings of interest by a change-of-coefficients argument.  We find it simpler to leave $\sfR$ general.
\end{remark}

\subsection{Categories}

By a ``dg category'' we will always mean $\sfR$-linear dg category.  If $C$ is a dg category, then $\dghom(x,y)$ denotes the chain complex of homomorphisms between objects $x$ and $y$ of $C$.  We will use $\Hom(x,y)$ to denote hom sets in non-dg settings.  We will regard the differentials in all chain complexes as having degree $+1$, i.e. $d:K^i \to K^{i+1}$.  If $K$ is a chain complex (of $\sfR$-modules or sheaves, usually) then $h^i(K)$ will denote its $i$th cohomology object.    We frequently abuse definitions and write  ``$K = 0$'' instead of ``$K$ is acyclic'' or ``$K$ is  quasi-isomorphic to $0$.''  We let $hC$ denote the $\sfR$-linear homotopy category of $C$, with the same objects and with $\Hom(x,y) = h^0(\dghom(x,y))$.

If $C$ is a dg category, then $\Tr(C)$ denotes the triangulated dg category generated by $C$, which is well-defined up to quasi-equivalence.  For example, $\Tr(C)$ may be taken to be the smallest full triangulated subcategory of $mod(C)$ containing the image of the Yoneda embedding $Y:C \to mod(C)$.  (See \cite{Dr,To} for the Yoneda embedding in the dg setting.)  Note that if $C$ is a full subcategory of a triangulated dg category $D$, then $\Tr(C)$ is naturally equivalent to the full triangulated subcategory $\langle C \rangle$ of $D$ generated by $C$---the composition of Yoneda maps $\langle C \rangle \hookrightarrow D \to mod(D) \to mod(C)$ provides such an equivalence.

\subsection{Schemes and coherent sheaves}

All schemes that appear will be over $\sfR$.  If $X$ is a scheme, then we let $\cQ(X)^\naive$ denote the dg category of bounded complexes of quasicoherent sheaves on $X$, and we let $\cQ(X)$ denote the dg localization (in the sense of \cite{Dr,To}) of this category with respect to acyclic complexes.  If $G$ is an algebraic group acting on $X$, we let $\cQ_G(X)^\naive$ denote the dg category of complexes of $G$-equivariant quasicoherent sheaves.  We let $\cQ_G(X)$ denote the localization of this category with respect to acyclic complexes.

We will only consider the case where $G$ is diagonalizable and $X$ admits a Zariski open cover by $G$-stable affine charts.  In that case, an object of $\cQ(X)$ (resp. $\cQ_G(X)$) is called \emph{strictly perfect} if it is quasi-isomorphic to a bounded complex of vector bundles (resp. equivariant vector bundles), and \emph{perfect} if this is true on each affine chart.  In Section \ref{sec:perfect} we will also make use of a monoidal criterion for perfectness, see Proposition \ref{prop:dualizable}.  We use $\Perf(X) \subset \cQ(X)$ and $\Perf_G(X)\subset \cQ_G(X)$ to denote the full triangulated dg subcategories consisting of perfect objects.  If $u:X \to Y$ is a morphism of schemes, we have natural dg functors $u_*:\cQ(X) \to \cQ(Y)$ and $u^*:\cQ(Y) \to \cQ(X)$.  Note that the functor $u^*$ carries $\Perf(Y)$ to $\Perf(X)$.

All algebraic groups that appear will be tori.  We write $\Gm$ for the multiplicative group over $\sfR$.
Suppose $G$ and $H$ are algebraic groups, $X$ is a scheme with a $G$-action, and $Y$ is a scheme with an $H$-action.  If a morphism $u:X \to Y$ is equivariant with respect to a homomorphism of algebraic groups $\phi:G \to H$, then we will often abuse notation and write $u_*$ and $u^*$ for the equivariant pushforward and pullback functors $u_*:\cQ_G(X) \to \cQ_H(Y)$ and $u^*:\cQ_H(Y) \to \cQ_G(X)$.  Again note that $u^*$ carries $\Perf_H(Y)$ to $\Perf_G(X)$.

\subsection{Constructible and microlocal geometry}

We refer to \cite{KS} for the theory of constructible sheaves, especially the microlocal aspects of this theory.  We refer to \cite{NZ,N,N2} and the references therein for details on how to extend this classical theory to the setting of dg categories.
If $X$ is a topological space we let $\Sh(X)$ denote the dg category of bounded chain complexes of sheaves of $\sfR$-modules on $X$, localized with respect to acyclic complexes.  If $X$ is a real-analytic manifold, then we will call an object $F \in \Sh(X)$ \emph{constructible} if it is constructible with respect to a real-analytic Whitney stratification of $X$, and if furthermore the stalk at every point is a perfect chain complex of $\sfR$-modules.  We let $\Sh_c(X)$ denote the full subcategory of constructible objects of $\Sh(X)$ and $\Sh_{cc}(X) \subset \Sh_c(X)$ is the full subcategory of constructible objects which have compact support.  If $f:X \to Y$ is a real-analytic map, we let $f_*$ and $f_!$ denote the usual push-forward and push-forward with compact support functors $\Sh(X) \to \Sh(Y)$, and $f^*$ and $f^!$ the usual pull-back and Verdier-pullback functors $\Sh(Y) \to \Sh(X)$.  Note that $f^*$ and $f^!$ always carry $\Sh_c(Y)$ to $\Sh_c(X)$, and $f_!$ and $f_*$ carry $\Sh_c(X)$ to $\Sh_c(Y)$ so long as the fibers of $f$ have finitely many connected components.  We write $\cD$ for the Verdier duality functor $\Sh_c(X)^\circ \stackrel{\sim}{\to} \Sh_c(X)$.

%We use $D_c(X)$ and $D_{cc}(X)$ to denote the derived category $D(Sh_c(X))$ and $D(Sh_{cc}(X))$ respectively.

We use the ``standard'' and ``costandard'' terminology from \cite{NZ} for certain sheaves extended from analytic submanifolds.  The \emph{standard constructible sheaf} on the submanifold $i_{Y}: Y\hookrightarrow X$ is defined as the push-forward of the constant sheaf on $Y$, i.e. $i_{Y*} \sfR_Y$ as an object in $\Sh_c(X)$. The Verdier duality functor $\cD: \Sh_c(X)^\circ \to \Sh_c(X)$ takes $i_{Y*}\sfR_Y$ to the \emph{costandard constructible sheaf} on $X$. We know $\cD(i_{Y*}\sfR_Y)=i_{Y!}\cD(\sfR_Y)=i_{Y!} \omega_Y$, where $\omega_Y=\cD(\sfR_Y)=\ori_Y[\dim Y]$ is the orientation sheaf on $Y$ placed in cohomological degree $-\dim Y$.

%We denote the singular support of a complex of sheaves $F$ by $\SS(F) \subset T^*X$.  If $X$ is a real-analytic manifold and $\Lambda \subset T^*X$ is an $\bR_{>0}$-invariant Lagrangian subvariety, then $\Sh_c(X;\Lambda)$ (resp. $\Sh_{cc}(X;\Lambda)$) denotes the full subcategory of $\Sh_c(X)$ (resp. $\Sh_{cc}(X)$)
%whose objects have singular support in $\Lambda$.

\subsection{Toric geometry}\label{sec:toric-geometry}

We refer to \cite{Fu} for the theory of toric varieties.  Let $N \cong \bZ^n$ be a free abelian group, and let $\Sigma$ be a \emph{fan} in $N$ (or in $N_\bR = N \otimes \bR$) of strongly convex rational polyhedral cones.  We do not necessarily assume that $\Sigma$ satisfies further conditions---e.g. that it is complete, or simplicial.

Given $N$ and $\Sigma$, we fix the following standard notation:
\begin{itemize}
\item  $M = \Hom(M,\bZ)$ is the dual lattice to $N$.
\item  $N_\bR$ and $M_\bR$ are the real vector spaces spanned by $N$ and $M$, i.e. $N_\bR = N \otimes_\bZ \bR$ and
$M_\bR = M \otimes_\bZ \bR$.
\item  $X = X_\Sigma$ is the toric variety associated to $\Sigma$, defined over the base ring $\sfR$.  It is naturally equipped with an action of the algebraic torus $T = N \otimes \Gm$.
\end{itemize}

We also use:
\begin{itemize}
%4\item  We let $T_\bR$ denote the maximal compact subgroup of $T$.  Thus $T_\bR \cong N_\bR/N\cong \mathrm{U}(1)^n$.
%\item  Similarly, define $T^\vee=M\otimes \bC^*$, and $T^\vee_\bR$ to be the maximal compact subgroup of it.
%Thus $T_\bR^\vee\cong M_\bR/M$.
%\item Let $\Si(d)$ denote the set of $d$-dimensional cones in $\Si$. Let
%$\Si(1)=\{ \rho_1,\ldots,\rho_r\}$ be the set of 1-dimensional cones in $\Si$,
%and let $v_i\in N$ be the generator of $\rho_i$, i.e. $\rho_i\cap N = \bZ_{\geq 0} v_i$.
\item Let $\sigma\in \Sigma$ be a $d$-dimensional cone. Let
\begin{eqnarray*}
\sigma^\vee &=&\{ x\in M_\bR\mid \langle x,v\rangle \geq 0 \textup{ for all } v\in \sigma\},\\
\sigma^\perp &=& \{x\in M_\bR\mid \langle x,v\rangle =0 \textup{ for all } v\in \sigma\},
\end{eqnarray*}
where $\langle \ , \ \rangle: M_\bR\times N_\bR\to \bR$ is the natural pairing.
Then $\sigma^\vee \subset M_\bR$ is the dual cone of $\sigma$, and $\sigma^\perp$ is a codimension
$d$ subspace of $M_\bR$. Let $M(\sigma)= \sigma^\perp\cap M$. Then $M(\sigma)$ is a rank $n-d$ sublattice of $M$.
\item We denote the interior of a cone by $\sigma^\circ$.  More precisely, for $\sigma \in \Sigma$, we have
$$\begin{array}{ccc}
\sigma^\circ & = & \{v \in \sigma \mid v \notin \tau \text{ for any }\tau \in \Sigma, \, \tau \subsetneq \sigma\}, \\
(\sigma^\vee)^\circ & = & \{x \in M_\bR \mid \langle x,v\rangle >0 \text{ for all }x \in \sigma^\circ\}.
\end{array}
$$
\end{itemize}

Now we mention some basic facts about line bundles.  For simplicity let us assume $\Sigma$ is complete.   Fix a total ordering $C_1,\ldots,C_v$ of the maximal cones in $\Sigma$.  Since $X_\Sigma$ is complete, each $C_i$ corresponds to a $T$-fixed point $x_i \in X$.

\begin{definition}[twisted polytope]\label{df:twisted-polytope}
Let $\Sigma$ be a complete fan.  A \emph{twisted polytope} for $\Sigma$ is an ordered $v$-tuple $\uchi = (\chi_1,\ldots,\chi_v)$ of elements of $M$ with the property that for any $1 \leq i \leq v$ and $1 \leq j \leq v$, the linear forms
$\langle \chi_i,-\rangle$ and $\langle \chi_j,-\rangle$ agree when restricted to $C_i \cap C_j$.
\end{definition}

The terminology is motivated by \cite{KT}.

\begin{theorem}
\label{thm:fulton3}
For each twisted polytope $\uchi = (\chi_1,\ldots,\chi_v)$, there is up to isomorphism a unique line bundle $\cO_X(\uchi)$ with the properties that the fiber of $T$ over the fixed $\sfR$-point $x_i$ is a free $\sfR$ module of rank one on which $T$ acts with weight $\chi_i$.  $\cO_X(\uchi)$ is ample precisely when $\uchi$ satisfy the following two conditions:
\begin{enumerate}
\item The set $\{\chi_1,\ldots,\chi_v\}$ is strictly convex, in the sense that its convex hull is strictly larger than the convex hull of any subset $\{\chi_{i_1},\ldots,\chi_{i_w}\}$.
\item The convex hull of $\{\chi_1,\ldots,\chi_v\}$ coincides with the set of all $\xi \in M_\bR$ satisfying
$$\langle\xi,\gamma\rangle \geq \langle\chi_i,\gamma\rangle \text{ for all $i$ and all $\gamma \in C_i$.}$$
\end{enumerate}
\end{theorem}

\begin{proof}
See \cite[Chapter 3]{Fu}.
\end{proof}

\begin{remark}
Let us indicate the meaning of condition (2) in the theorem.  By definition, a polytope is an intersection of half-spaces---one half-space for each top-dimensional face of the polytope.  At each vertex of this polytope, we may construct a cone by taking the intersection of only those half-spaces which contain the vertex (equivalently, only those half-spaces corresponding to a face incident with the vertex).  Condition (2) states that the cone so obtained from the vertex $\chi_i$ is some translate of the dual  cone of $C_i$.
\end{remark}

\section{Categories of $\Theta$-sheaves}
\label{sec:three}

\subsection{The poset $\bGamma(\Sigma,M)$}

Let $\bGamma(\Sigma,M)$ denote the set of pairs $(\sigma,\chi)$ where $\sigma$ is a cone in $\Sigma$ and $\chi$ is an integral coset of $\sigma^\perp$---i.e. a coset that passes through a lattice vector.  We endow $\bGamma(\Sigma,M)$ with a partial order by setting $(\sigma,\phi) \leq (\tau,\psi)$ whenever either of the following equivalent conditions are satisfied:
\begin{itemize}
\item $\tau \subset \sigma$, and if $\overline{\phi}$ denotes the image of $\phi$ in $M_\bR/\tau^\perp$ then $\overline{\phi} - \psi \in \tau^\vee$.
\item The subsets $\phi + \sigma^\vee \subset M_\bR$ and $\psi + \tau^\vee \subset M_\bR$ have $\phi + \sigma^\vee \subset \psi + \tau^\vee$.
\end{itemize}

We may regard the poset $\bGamma(\Sigma,M)$ as a category in the usual way.  Let $\bGamma(\Sigma,M)_\sfR$ denote the $\sfR$-linearization of this category.  More precisely, $\bGamma(\Sigma,M)_\sfR$ is the dg category defined in the following way:
\begin{itemize}
\item The objects of $\bGamma(\Sigma,M)_\sfR$ are the elements of $\bGamma(\Sigma,M)$.
\item The hom complexes are concentrated in degree zero, and are given by
$$\dghom((\sigma,\phi),(\tau,\psi)) = \bigg\{ \begin{array}{cl} \sfR & \text{if $(\sigma,\phi) \leq (\tau,\psi)$,} \\ 0 & \text{otherwise.}  \end{array}$$
\item The composition map
$$\dghom((\tau,\psi),(\upsilon,\chi))\otimes_\sfR \dghom((\sigma,\phi),(\tau,\psi)) \to \dghom((\sigma,\phi),(\upsilon,\chi))$$
is given by $1 \otimes_{\sfR} 1 \mapsto 1$ if $(\sigma,\phi) \leq (\tau,\psi)$ and $(\tau,\psi) \leq (\upsilon,\chi)$, and is the zero map otherwise.
\end{itemize}

\subsection{The sheaves $\Theta$ and $\Theta'$}

In this section we define the objects $\Theta(\sigma,\chi) \in \Sh_c(M_\bR)$ and $\Theta'(\sigma,\chi) \in \cQ_T(X)$ indexed by $\bGamma(\Sigma,M)$.

\begin{definition}
\label{def:theta}
Let $(\sigma,\chi) \in \bGamma(\Sigma,M)$.  Let $j$ denote the inclusion map $j:\chi + \sigma^\perp \hookrightarrow M_\bR$.  Define $\Theta(\sigma,\chi) \in \Sh_c(M_\bR)$ to be the costandard sheaf on $M_\bR$ associated to the open set $(\chi+ \sigma^\vee)^\circ \subset M_\bR$, i.e.
$$\Theta(\sigma,\chi) = j_{(\chi + \sigma^\vee)^\circ !} \omega_{(\chi + \sigma^\vee)^\circ}$$
\end{definition}

To define the quasicoherent counterparts of the sheaves $\Theta(\sigma,\chi)$, let 
$$X_\sigma = \Spec \sfR[\sigma^\vee \cap M]$$
 be the $T$-stable affine open subset of $X$ corresponding to $\sigma \in \Sigma$.  To give an equivariant quasicoherent sheaf on $X_\sigma$ is equivalent to giving an $M$-graded $\sfR[\sigma^\vee \cap M]$-module.  Let $\cO_\sigma(\chi)$ denote the quasicoherent sheaf associated to the module $\sfR[(\chi + \sigma^\vee) \cap M]$ with the obvious $M$-grading.

\begin{definition}
\label{def:thetap}
Let $(\sigma,\chi) \in \bGamma(\Sigma,M)$.  Let $j$ denote the inclusion map $X_\sigma \hookrightarrow X$.  Define
$$\Theta'(\sigma,\chi) = j_* \cO_\sigma(\chi)$$
\end{definition}

The sheaves $\{\Theta(\sigma,\chi)\}$ and $\{\Theta'(\sigma,\chi)\}$ have identical (and very simple) patterns of Ext-groups:

\begin{proposition}
\label{prop:exts}
Let $(\sigma,\phi)$ and $(\tau,\psi)$ be elements of $\bGamma(\Sigma,M)$.  Then
\begin{enumerate}
\item
$$
\Ext^i(\Theta(\sigma,\phi),\Theta(\tau,\psi)) \cong \bigg\{ \begin{array}{cl}
\sfR & \text{if $i = 0$ and $(\sigma,\phi) \leq (\tau,\psi)$,} \\
0 & \text{otherwise.}
\end{array}
$$
\item
$$\Ext^i(\Theta'(\sigma,\phi),\Theta'(\tau,\psi)) \cong \bigg\{\begin{array}{cl}
\sfR & \text{if $i = 0$ and $(\sigma,\phi) \leq (\tau,\psi)$,} \\
0 & \text{otherwise.}
\end{array}$$
\end{enumerate}
\end{proposition}

\begin{proof}
By tensoring $\Theta(\sigma,\phi)$ and $\Theta(\tau,\psi)$ with $\ori[-\dim(M_\bR)]$, we have
$$
\Ext^i(\Theta(\sigma,\phi),\Theta(\tau,\psi)) = \Hom(j_{(\phi+\sigma^\vee)^\circ !} \sfR,j_{(\psi+\tau^\vee)^\circ !}\sfR[i]).
$$
By the standard adjunction this is isomorphic to
$$
\Hom(\sfR,j_{(\phi+\sigma^\vee)^\circ}^! j_{(\psi+\tau^\vee)^\circ !}\sfR[i]) =
H^i(
(\phi+\sigma^\vee)^\circ,j_{(\phi+\sigma^\vee)^\circ}^! j_{(\psi+\tau^\vee)^\circ !}\sfR).
$$
The sheaf $j_{(\phi+\sigma^\vee)^\circ}^! j_{(\psi+\tau^\vee)^\circ !}\sfR$ is the extension-by-zero
of the constant sheaf on $(\phi+\sigma^\vee)^\circ \cap (\psi+\tau^\vee)^\circ$.
Setting $A = (\phi+\sigma^\vee)^\circ - (\phi+\sigma^\vee)^\circ \cap (\psi+\tau^\vee)^\circ$, the sheaf cohomology above is isomorphic to the relative cohomology
$$
H^i((\phi+\sigma^\vee)^\circ,A)
$$
Since $\sigma^\vee$ and $\tau^\vee$ are convex, there are two possibilities.
Either $\phi+\sigma^\vee \subset \psi+\tau^\vee$, in which case $A$ is empty, or else $A$ is contractible.  

This completes the proof of part (1).
Let us now prove part (2).  The Ext group is isomorphic to $\Hom(j_{\sigma*} \cO_\sigma(\phi),j_{\tau*} \cO_\tau(\psi)[i])$, which by adjunction is isomorphic to
$$\Hom(j_\tau^* j_{\sigma *} \cO_\sigma(\phi), \cO_\tau(\psi)[i]).$$
Since $X_\tau$ is affine, this group vanishes when $i \neq 0$, as required.  Let us consider the case $i = 0$.

Note that we have $j_\tau^* j_{\sigma *} \cO_\sigma(\phi)$ is isomorphic to $i_*\cO_{\sigma \cap \tau}(\phi)$, where $i$ denotes the inclusion of $X_{\sigma \cap \tau}$ into $X_\tau$.  This cannot map into
$\cO_\tau(\psi)$, or into any finitely generated $\sfR[\tau^\vee \cap M]$-module, unless $(\sigma \cap \tau)^\vee = \tau^\vee$, or equivalently unless $\tau \subset \sigma$.  In that case we have
$$
\Hom(i_* \cO_{\sigma \cap \tau}(\phi),\cO_\tau(\psi)) = \Hom(\cO_\tau(\phi),\cO_\tau(\psi))
$$
which is $\sfR$ whenever $\phi -\psi \in \tau^\vee$ and $0$ otherwise.  This completes the proof.
\end{proof}

In fact there are canonical generators of $\Hom(\Theta(\sigma,\chi),\Theta(\tau,\psi))$ and $\Hom(\Theta'(\sigma,\chi),\Theta'(\tau,\psi))$ whenever these groups are nonvanishing.  The composition
maps
$$\begin{array}{rcl}
\Hom(\Theta(\tau,\psi),\Theta(\upsilon,\chi)) \otimes_\sfR \Hom(\Theta(\sigma,\phi),\Theta(\tau,\psi)) & \to &\Hom(\Theta(\sigma,\phi),\Theta(\upsilon,\chi))\\
\Hom(\Theta'(\tau,\psi),\Theta'(\upsilon,\chi)) \otimes_\sfR \Hom(\Theta'(\sigma,\phi),\Theta'(\tau,\psi)) & \to & \Hom(\Theta'(\sigma,\phi),\Theta'(\upsilon,\chi))
\end{array}
$$ 
take canonical generators to canonical generators.

\begin{theorem}
\label{thm:mainfirst}
Let $X$ be a toric variety with fan $\Sigma$.
Let $\ltr \subset \Sh_c(M_\bR)$ denote the full triangulated subcategory generated by the objects $\Theta(\sigma,\chi)$.  Let $\ltrp \subset \cQ_T(X)$ denote the full triangulated subcategory generated by the objects $\Theta'(\sigma,\chi)$.  There exists a quasi-equivalence of dg categories $\kappa:\ltrp \to \ltr$ with the following properties:
\begin{itemize}
\item $\kappa(\Theta'(\sigma,\chi)) \cong \Theta(\sigma,\chi)$.
\item If $(\sigma,\phi) \leq (\tau,\psi)$ in $\bGamma(\Sigma,M)$, then the map $\Ext^0(\Theta'(\sigma,\phi),\Theta'(\tau,\psi)) \to \Ext^0(\Theta(\sigma,\phi),\Theta(\tau,\psi))$ induced by $\kappa$ carries the canonical generator of the source to the canonical generator of the target.
\end{itemize}
\end{theorem}

\begin{proof}
By Proposition \ref{prop:exts}, there is a dg functor
$\bGamma(\Sigma,M)_\sfR \hookrightarrow \Sh(M_\bR)$ that carries $(\sigma,\chi)$ to $\Theta(\sigma,\chi)$, and such that the map of chain complexes
$$\dghom_{\bGamma(\Sigma,M)_\sfR}((\sigma,\phi),\Theta(\tau,\psi)) \to \dghom_{\Sh_c(M_\bR)}(\Theta(\sigma,\phi),\Theta(\tau,\psi))$$
maps the canonical generator of $h^0$ of the source to the canonical generator of $h^0$ of the target.  In particular, this functor is a full dg embedding.  By the same reasoning there is a full dg embedding $\bGamma(\Sigma,M)_\sfR \hookrightarrow \cQ_T(M_\bR)$ with the same properties.  It follows that we have a diagram of quasiequivalences
$$\ltrp \stackrel{\sim}{\leftarrow} \Tr(\bGamma(\Sigma,M)_\sfR) \stackrel{\sim}{\to} \ltr$$
We let $\kappa$ be the composition.
\end{proof}

In Section  \ref{sec:shardsheaves} and Section \ref{sec:finfib}, we show that the categories $\ltr$ and $\ltrp$ can be described more intrinsically, instead of by generators.

\begin{corollary}
\label{cor:perfect}
Let $X$ be a toric variety corresponding to a fan $\Sigma$.  The dg functor $\kappa$ defines a full embedding of $\Perf_T(X)$ into $\Sh_c(M_\bR)$.
\end{corollary}

\begin{proof}
We just have to show that $\Perf_T(X)$ belongs to the domain of $\kappa$---i.e. that $\Perf_T(X_\Sigma) \subset \ltrp$.  First note that on any affine toric variety, a $T$-equivariant vector bundle splits as a sum of line bundles, necessarily of the form $\cO(\chi_\alpha)$.  On an affine $T$-variety a perfect complex is strictly perfect---it is quasi-isomorphic to a bounded complex of equivariant vector bundles.  It follows that if $j:U \hookrightarrow X$ is the inclusion of a $T$-stable affine chart, and $\cE$ is a $T$-equivariant perfect complex on $U$, then $j_* \cE$ belongs to $\ltrp$.

Now let $\cF$ be a perfect complex on $X$.  Fix a total order on the set of maximal cones $C_1,\ldots,C_v$ in $\Sigma$, and put $C_{i_0\cdots i_k} := C_{i_0} \cap \cdots \cap C_{i_k}$.  We have a \v Cech resolution of $\cF$---the complex $\cF$ is quasi-isomorphic to the total complex of the bicomplex
$$
\bigoplus_{i_0} j_{C_{i_0}*} \cF \vert_{X_{C_{i_0}}} \to \bigoplus_{i_0 < i_1} j_{C_{i_0 i_1}*} \cF\vert_{X_{C_{i_0 i_1}}} \to \cdots
$$
As each summand of each term of this complex belongs to $\ltrp$, the horizontal length of this bicomplex is $\leq v$, and each direct sum is finite, this establishes that $\cF \in \ltrp$.
\end{proof}

\subsection{Coherent-constructible dictionary---line bundles}

In this section we consider the case where $X$ is a proper toric variety and the corresponding fan $\Sigma$ is complete (i.e. the support of $\Sigma$ is all of $N_\bR$.)  We will compute $\kappa(\cL)$ when $\cL$ is an equivariant line bundle on $X$.

Let us make the following notation.  We let $C_1,\ldots,C_v$ and $C_{i_0 \cdots i_k}$ be as in the proof of Corollary \ref{cor:perfect}.  If $\uchi = (\chi_1,\ldots,\chi_v)$ is a twisted polytope,
for each $i_0 < \ldots < i_k$, let $\chi_{i_0\cdots i_k} \subset M_\bR$ be the affine hull of $\chi_{i_0},\cdots,\chi_{i_k}$.  Note that $\chi_{i_0 \cdots i_k}$ is a coset of $(C_{i_0\cdots i_k})^\perp$.  Whenever $j_0 < j_1 < \ldots < j_{\ell}$ refines $i_0 < \ldots< i_k$, we have a canonical map
$\Theta(C_{i_0 \ldots i_k},\chi_{i_0 \ldots i_k}) \to \Theta(C_{j_0 \cdots j_\ell},\chi_{j_0 \ldots j_\ell})$.  

\begin{definition}
For each twisted polytope $\uchi$, define $P(\uchi) \in \Sh_c(M_\bR)$ to be the cochain complex
$$
\bigoplus_{i_0} \Theta(C_{i_0},\chi_{i_0})
\to
\bigoplus_{i_0 < i_1} \Theta(C_{i_0 i_1},\chi_{i_0 i_1})
\to
\cdots
$$
\end{definition}

We regard the first term of the complex defining $P(\uchi)$ to be in degree zero, but because $\Theta(\sigma,\chi) = j_! \ori[\dim(M_\bR)]$, this means that $P(\uchi)$ is isomorphic to a complex of sheaves whose first term is in degree $-\dim(M_\bR)$.

Let $\cO_X(\uchi)$ denote the line bundle on $X$ associated to $\uchi$.  From the \v Cech complex discussed in the proof of Corollary \ref{cor:perfect}, it is clear that $\kappa(\cO_X(\uchi)) \cong P(\uchi)$.  When $\cO(\uchi)$ is ample we can describe $P(\uchi)$ explicitly:

\begin{theorem}
\label{thm:linbun}
Let $X$ be a complete toric variety, and let $P(\uchi)\in \Sh(M_\bR)$ be as in the previous definition.
Denote the convex hull of $\{\chi_1,\ldots,\chi_v\}$ by $\bfP$ and its interior by $\bfP^\circ$.
If $\cO_X(\uchi)$ is ample, then $P(\uchi) \cong j_!(\omega_{\bfP^\circ}) = j_!(\ori_{\bfP^\circ}[\dim(M_\bR)]$, where $j: \bfP^\circ \hookrightarrow M_\bR$ denotes the inclusion map.
\end{theorem}

\begin{proof}
Since $\bfP^\circ$ is an open subset of each $\chi_i + (C_i^\vee)^\circ$, we have a canonical map $j_! \omega_{\bfP^\circ} \to \Theta(C_i,\chi)$ for each $i$.  These maps assemble to a map of chain complexes $j_! \omega_{\bfP^\circ} \to P(\uchi)$.  To show that this map is a
quasi-isomorphism it suffices to show that it is a quasi-isomorphism on stalks.  That is,
we have to show that for each $x \in M_\bR$ the sequence
$$
\big(\bigoplus_{i_0} \Theta(C_{i_0},\chi_{i_0})\big)_x
\to
\big(\bigoplus_{i_0 < i_1} \Theta(C_{i_0 i_1},\chi_{i_0 i_1})\big)_x
\to
\cdots
$$
is acyclic when $x$ is not in the interior of the convex hull of $\uchi$ and concentrated in degree zero otherwise.

Let $\Delta$ be the combinatorial simplex on the numbers $\{1,\ldots,v\}$ and let $K(x)$ be the subsimplicial complex given by
$$\{i_0,\ldots,i_k\} \in K(x) \text{ if and only if }x \notin (\chi_{i_0\cdots i_k} + C_{i_0 \cdots i_k}^\vee)^\circ.$$
The cochain complex above is isomorphic to the relative simplicial cochain complex of the pair
$(\Delta,K(x))$.  Thus the $i$th cohomology of the complex is the relative cohomology group $H^i(\Delta,K(x))$.  When $x$ is in the interior of the convex hull, $K(x)$ is empty, so we have
$H^0 = \sfR$ and $H^i = 0$ for $i >0$ as required.

To complete the proof we have to show that the geometric realization $|K(x)|$ is contractible when $x$ is not in the interior of the convex hull.  Note that $K(x)$ is closely related to the union of faces of the convex hull which are ``visible'' from $x$.  In fact $\{i_0,\ldots,i_k\}$ is a $k$-simplex in $K(x)$ if and only if there is a line segment from $x$ to a face of $\bfP$ containing $\chi_{i_0},\cdots,\chi_{i_k}$ that does not pass through any point of $\bfP^\circ$.  Let us denote this union of faces by $K'$.

$K(x)$ is precisely the \v Cech nerve of the open cover of $K'$ whose charts are the ``stars'' of the vertices belonging to $K'$.  That is, the open subset of $K'$ corresponding to the vertex $w$ is the union of faces in $K'$ which contain $w$.  Note that $K'$ is homeomorphic to the projection of $\bfP$ onto any hyperplane lying between $x$ and $\bfP$, which will necessarily be convex and therefore contractible.  Thus $K'$ and $|K(x)|$ are contractible.

\end{proof}

\subsection{Coherent-constructible dictionary---functoriality and tensoriality}

In this section we show that the equivalence $\kappa:\ltrp \to \ltr$ between quasicoherent and constructible sheaves intertwines with appropriate pull-back and push-forward functors.  As a consequence it also intertwines the tensor product of quasicoherent sheaves (belonging to $\ltrp$) with the convolution product of constructible sheaves.

Let $N_1$ and $N_2$ be two lattices, let $\Sigma_1$ be a rational polyhedral fan in $N_{1,\bR}$ and let $\Sigma_2$ be a rational polyhedral fan in $N_{2,\bR}$.  For $i = 1$ or $2$, write $M_i, M_{i;\bR}, T_i,X_i,\ldots$ for the lattice dual to $N_i$, the real vector space dual to $N_{i;\bR}$, the algebraic torus $N_i \otimes \Gm$, the toric variety $X_{\Sigma_i}$, etc.  Write $\ltr_i$ (resp. $\ltrp_i$) for the subcategory of $\Sh_c(M_{i,\bR})$ (resp. $\cQ_{T_i}(X_i)$)  generated by the $\Theta(\sigma,\chi)$ (resp. $\Theta'(\sigma,\chi)$),  with $(\sigma, \chi) \in \bGamma(\Si_i)$.

Let us say that a map $f:N_1 \to N_2$ is \emph{fan-preserving} if for each $\sigma_1 \in \Sigma_1$, the image of $\sigma_1$ under $N_{1,\bR} \to N_{2,\bR}$ lies in another cone $\sigma_2 \in \Sigma_2$.  Note that $f$ induces
\begin{itemize}
\item a map $f \otimes 1_{\Gm}:T_1 \to T_2$.

\item a map $f^\vee:\sigma_2^\vee \cap M_2 \to \sigma_1^\vee \cap M_1$ for each pair of cones $\sigma_1 \in \Sigma_1$, $\sigma_2 \in \Sigma_2$ with $f(\sigma_1) \subset \sigma_2$.

\item a map $u_{f,\sigma_1,\sigma_2}:X_{\sigma_1} \to X_{\sigma_2}$ for a pair of cones $\sigma_1,\sigma_2$ as above.  (As usual $X_\sigma$ denotes the affine $T$-invariant chart of $X$ corresponding to the cone $\sigma$)
\item a map $u = u_f:X_1 \to X_2$ assembled from the $u_{f,\sigma_1,\sigma_2}$, equivariant with respect to $f \otimes 1_{\Gm}:T_1 \to T_2$.
\end{itemize}

More straightforwardly, a fan-preserving map induces a map of real vector spaces $M_{2,\bR} \to M_{1,\bR}$; we will denote this map by $v = v_f$.

\begin{theorem}[functoriality]
\label{thm:functoriality}
Let $f$ be a fan-preserving map from $\Sigma_1 \subset N_{1,\bR}$ to $\Sigma_2 \subset N_{2,\bR}$.  Suppose that $f$ furthermore satisfies the following conditions:
\begin{enumerate}
\item the inverse image of any cone $\sigma_2 \subset \Sigma_2$ is a union of cones in $\Sigma_1$.  (For instance, if both fans are complete then $f$ automatically satisfies this condition.)
\item $f$ is injective.
\end{enumerate}
Let $u$ and $v$ be as above.  Then
\begin{enumerate}
\item The pullback $u^*:\cQ_{T_2}(X_2) \to \cQ_{T_1}(X_1)$ takes $\ltrp_2$ to $\ltrp_1$.

\item The proper pushforward $v_!:Sh_c(M_{2,\bR}) \to Sh_c(M_{1,\bR})$ takes $\ltr_2$ to $\ltr_1$.

\item The following square of functors commutes up to natural isomorphism:
$$\xymatrix{
\ltrp_2 \ar[r]^{\kappa_2} \ar[d]_{u^*} &\ltr_2 \ar[d]^{v_!} \\
\ltrp_1 \ar[r]_{\kappa_1} &  \ltr_1
}$$

\end{enumerate}
\end{theorem}

\begin{remark}
In the larger categories $\cQ$ and $\Sh_c$, the functors $u^*$ and $v_!$ have right adjoints $u_*$ and $v^!$.  However these adjoint functors do not usually preserve the categories $\ltrp$ and $\ltr$; we therefore do not understand very well how $\kappa$, $u_*$, and $v^!$ interact.
\end{remark}

\begin{proof}
For each cone $\sigma_2 \in \Sigma_2$, we have a cartesian square
$$\xymatrix{
u^{-1}(X_{\sigma_2}) \ar[r] \ar[d] & X_{\sigma_2} \ar[d] \\
X_1 \ar[r]_{u_f} & X_2
}
$$
The vertical arrows are open inclusions, so we may apply the flat base change formula.  Fix $\chi_2 \in M$ and set $\chi_1 = v(\chi_2)$.  The flat base change formula gives us $u_f^* \Theta'(\sigma_2,\chi_2 + \sigma_2^\perp) \cong j_{u^{-1}(X_{\sigma_2})*} \cO(\chi_1)$, where $\cO(\chi_1)$ denotes the structure sheaf equipped with equivariant structure given by $\chi_1$.   Fix a total order on the set of maximal cones $B_1,\ldots,B_w$ contained in $f^{-1}(\sigma_2)$, and put $B_{i_0\cdots i_k} = B_{i_0} \cap \cdots \cap B_{i_k}$.  By applying $j_{u^{-1}(X_{\sigma_2})*}$ to the \v Cech complex for $\cO(\chi_1)$, we get a quasi-isomorphism between $u_f^* \Theta'(\sigma_2,\chi_2)$ and the complex
$$\bigoplus_{i_0} \Theta'(B_{i_0},\chi_1+ B_{i_0}^\perp) \to \bigoplus_{i_0 < i_1} \Theta'(B_{i_0 i_1},\chi_1 + B_{i_0 i_1}^\perp) \to \cdots$$
In particular, this proves the first assertion.

After Theorem \ref{thm:mainfirst}, the second and third assertions follow from the commutativity of the following diagram:
$$\xymatrix{
\ltrp_2 \ar[r]^{\kappa_2} \ar[d]_{u^*} &\Sh_c(M_{2,\bR}) \ar[d]^{v_!} \\
\ltrp_1 \ar[r]_{\kappa_1} &  \Sh_c(M_{1,\bR})
}$$

To construct a natural quasi-isomorphism
$\iota:v_! \circ \kappa_2 \stackrel{\sim}{\to} \kappa_1 \circ u^*$.
it suffices to give maps
$$\iota_{\sigma_2,\chi_2}:v_{!} \kappa_2(\Theta'(\sigma_2,\chi_2+ \sigma_2^\perp)) \to \kappa_1(u^*(\Theta'(\sigma_2,\chi_2 + \sigma_2^\perp)))$$
with the following properties:
\begin{itemize}
\item Each map $\iota$ is a quasi-isomorphism.
\item Whenever $(\sigma_2,\chi_2 + \sigma_2^\perp) \leq (\tau_2,\psi_2 + \tau_2^\perp)$ in $\bGamma(\Sigma,M)$, the following square commutes:
$$
\xymatrix{
v_! \kappa_2(\Theta'(\sigma_2,\chi_2 + \sigma_2^\perp)) \ar[r] \ar[d] & v_! \kappa_2(\Theta'(\tau_2,\psi_2 + \tau_2^\perp)) \ar[d]\\
\kappa_1(u^* \Theta'(\sigma_2,\chi_2 + \sigma_2^\perp)) \ar[r] & \kappa_1(u^*\Theta'(\tau_2,\psi_2 + \tau_2^\perp))
}
$$
where the vertical arrows are given by $\iota$ and the horizontal arrows are induced by the canonical morphism $\Theta'(\sigma_2,\chi_2+ \sigma_2^\perp) \to \Theta'(\tau_2,\psi_2 + \tau_2^\perp)$.

\end{itemize}

Let us compute $v_! \kappa_2(\Theta'(\sigma_2,\chi_2 + \sigma_2^\perp))$.  In fact we have a quasi-isomorphism
$$v_! \kappa_2(\Theta'(\sigma_2,\chi_2 + \sigma_2^\perp)) \stackrel{\sim}{\to} j_{(v(\chi_2) + v(\sigma_2^\vee))^\circ!} \omega_{(\chi_2 + \sigma_2^\vee)^\circ}$$
To see this, first note that whenever $m:U \to V$ is a submersion, there is a natural map $m_! \omega_U \to \omega_V$, given by composing the isomorphism $m^! \omega_V \cong \omega_U$ with the adjunction map.  By definition we have $\kappa_2(\Theta'(\sigma_2,\chi_2 + \sigma_2^\perp)) = \Theta(\sigma_2,\chi_2 + \sigma_2^\perp) = j_{(\chi_2 + \sigma_2^\vee)^\circ} \omega$.  Since $f$ is injective, $v_f:M_{2,\bR} \to M_{1,\bR}$ is surjective, and its restriction to $(\chi_2 + \sigma_2^\vee)^\circ$ is a submersion whose fibers are open balls.  It follows that the canonical map
$$v_{f!} \Theta(\sigma_2,\chi_2 + \sigma_2^\perp) \to j_{(v(\chi_2) + v(\sigma_2^\vee))^\circ!} \omega_{(\chi_2 + \sigma_2^\vee)^\circ}$$
is a quasi-isomorphism.

Now let us compute $\kappa_1 u^* \Theta'(\sigma_2,\chi_2 + \sigma_2^\perp)$.  Put $\chi_1 = v(\chi_2)$.  We have already seen that $u^*\Theta'(\sigma_2,\chi_2)$ has the \v Cech resolution
$$\bigoplus_{i_0} \Theta'(B_{i_0},\chi_1 + B_{i_0}^\perp) \to \bigoplus_{i_0 < i_1} \Theta'(B_{i_0 i_1},\chi_1 + B_{i_0 i_1}^\perp) \to \cdots$$
After applying $\kappa_1$ we have
$$\bigoplus_{i_0} \Theta(B_{i_0},\chi_1 + B_{i_0}^\perp) \to \bigoplus_{i_0 < i_1} \Theta(B_{i_0 i_1},\chi_1 + B_{i_0 i_1}^\perp) \to \cdots.
$$
Now we define the map $\iota:v_! \kappa_2(\Theta'(\sigma_2,\chi_2 + \sigma_2^\perp)) \to \kappa_1 u^* \Theta'(\sigma_2,\chi_2 + \sigma_2^\perp)$ to be the morphism of complexes:
$$
\begin{CD}
j_{(v(\chi_2) + v(\sigma_2^\vee))^\circ !}\omega   @>>>   0 @>>>  \cdots \\
@VVV @VVV\\
\bigoplus_{i_0} \Theta(B_{i_0},\chi_1 + B_{i_0}^\perp) @>>> \bigoplus_{i_0 < i_1} \Theta(B_{i_0 i_1},\chi_1 + B_{i_0 i_1}^\perp) @>>> \cdots
\end{CD}
$$
where the nonzero vertical arrow is the direct sum of the maps induced by the inclusion of open sets
$$(\chi_1 + v(\sigma_2^\vee))^\circ \subset \chi_1 + B_{i_0}^\circ$$
This map has the desired naturality property.  It remains to show that it is a quasi-isomorphism.
We do this at the level of stalks, following the proof of Theorem \ref{thm:linbun}.  We have to show that for each point $a \in M_{1,\bR}$, the cohomology of the chain complex
$$
\bigoplus_{i_0} (\Theta(B_{i_0},\chi_1 + B_{i_0}^\perp))_a \to \bigoplus_{i_0 < i_1} (\Theta(B_{i_0 i_1},\chi_1 + B_{i_0 i_1}^\perp))_a \to \cdots
$$
is a one-dimensional vector space concentrated in degree $-\dim(M_{1,\bR})$, and zero otherwise.  By replacing $a$ by $a - \chi_1$ if necessary, we may assume that $\chi_1 = 0$.

Let $\Delta$ be the combinatorial simplex on the numbers $\{1,\ldots,w\}$ and let $K(a)$ be the subsimplicial complex given by
$$
\{i_0,\ldots,i_k\} \in K(a) \text{ if and only if } a \notin (B_{i_0,\ldots,i_k}^\vee)^\circ
$$
The cochain complex of stalks at $a$ is isomorphic to the relative simplicial complex of the pair $(\Delta,K(a))$, shifted by $\dim(M_{1,\bR})$.  When $a \in (v(\sigma_2^\vee))^\circ$, $K(a)$ is empty so we have $H^0 = \sfR$ and $H^i = 0 $ for $i > 0 $ as required.

To complete the proof we have to show that the geometric realization $|K(a)|$ is contractible when $a$ is not in $(v(\sigma_2^\vee))^\circ$.  Let $E$ be the union of the interiors of those cones in $f^{-1}(\sigma_2)$ that do not lie entirely in $a^{-1}(\bR_{>0})$. (Here we are regarding $a$ as a linear function on $N_\bR$.)
Note that $E$ is convex.  We have $a \notin (B_{i_0 \cdots i_k}^\vee)^\circ$ if and only if $B_{i_0 \cdots i_k} \subset E$.  The contractibility of $|K(a)|$ now follows from the contractibility of $S \cap E$, where $S$ is the unit sphere in $E$.  Indeed, $K(a)$ is the \v Cech nerve of the open cover of $S \cap E$ given by stars---each simplex $\{i_0,\ldots,i_k\}$ is associated to the open set $S \cap E \cap \mathrm{star}(B_{i_0,\ldots,i_k})$, where $\mathrm{star}(\tau)$ denotes the union of the interiors of cones containing $\tau$ as a face.

\end{proof}

\begin{example}
\label{ex:frobenius}
For any fan $\Sigma \subset N_\bR$, and for each positive integer $p$, the multiplication-by-$p$ map $N \to N$ satisfies the hypotheses of Theorem \ref{thm:functoriality}.  The corresponding map $u_p:X_\Sigma \to X_\Sigma$ is closely related to the Frobenius map when $p$ is prime and $X$ is defined over a field of characteristic $p$.  The map $v_p:M_\bR \to M_\bR$ is multiplication-by-$p$ again.  At the level of $K$-theory, this connection between the Frobenius map and dilation by $p$ is well-known, and plays a crucial role in the results of \cite{M}.
\end{example}

\begin{example}
\label{ex:blowup}
If $\Sigma_1$ and $\Sigma_2$ are fans in the same vector space $N_\bR$, and $\Sigma_1$ is a refinement of $\Sigma_2$, then the identity map $N_\bR \stackrel{=}{\to} N_\bR$ satisfies the hypotheses of Theorem \ref{thm:functoriality}.
The corresponding map $u:X_1 \to X_2$ is birational---toric resolutions of singularities arise in this way.  The corresponding map $v:M_\bR \to M_\bR$ is also the identity map; the pushforward functor $v_!$ is just an inclusion of categories.  Thus by
Theorem \ref{thm:functoriality} $\kappa \circ u^* = \kappa$.
\end{example}

\begin{example}
\label{ex:diagonal}
The diagonal map $N \to N \oplus N$ satisfies the hypotheses of Theorem \ref{thm:functoriality}.
The corresponding map $u:X \to X \times X$ is also the diagonal map, and the corresponding map $v:M_\bR \times M_\bR \to M_\bR$ is the addition map.
\end{example}

Recall that the \emph{convolution} of two sheaves $F$ and $G$ on a vector space $M_\bR$ is given by the formula $F \star G = v_!(F \boxtimes G)$, where $v$ denotes the addition map as in the example.  Convolution defines a monoidal structure on $\Sh(M_\bR)$ and various subcategories, including $\ltr$ and $\Sh_{cc}(M_\bR)$.  From Example \ref{ex:diagonal} we see the following

\begin{corollary}
\label{cor:tensor}
The functor $h\ltrp \stackrel{\sim}{\to} h\ltr$ induced by $\kappa$ has a natural monoidal structure which intertwines the tensor product on $h\ltrp$ with the convolution product on $h\ltr$.
\end{corollary}

We are grateful to Zhiwei Yun for suggesting this statement to us.

\begin{proof}
We have a natural isomorphism $\cF \otimes \cG \cong u^* (\cF \boxtimes \cG)$, and the convolution product $\kappa(\cF) \star \kappa(\cG) = v_!(\kappa(\cF) \boxtimes \kappa(\cG))$ by definition.  Thus by Theorem \ref{thm:functoriality} (functoriality) to construct a natural isomorphism $\kappa(\cF \otimes \cG) \cong \kappa(\cF) \star \kappa(\cG)$, it suffices to
construct a natural isomorphism
$$\kappa(\cF \boxtimes \cG) \cong \kappa(\cF) \boxtimes \kappa(\cG)$$
We may assume that $\cF$ is of the form $\Theta'(\sigma,\chi)$ and $\cG$ is of the form $\Theta'(\tau,\psi)$.  But in that case it is easy to construct natural isomorphisms $\cF \boxtimes \cG \cong \Theta'(\sigma \times \tau, (\chi,\psi))$ and $\kappa(\cF) \boxtimes \kappa(\cG) \cong \Theta(\sigma \times \tau, (\chi,\psi))$.  To verify that $\kappa$ commutes with the associativity isomorphisms $(\cF \otimes \cG) \otimes \cH$ we may apply Theorem \ref{thm:functoriality} to the isomorphism $(X \times X) \times X \cong X \times (X \times X)$. 
\end{proof}

\begin{remark}
Recall that monoidal structures on dg categories have a hierarchy of commutativity constraints $E_2,E_3,\ldots,E_\infty$.  By applying
Theorem \ref{thm:functoriality} (functoriality) to more general diagonal maps such as $X \to X \times X \times X$ and to the permutation maps
$X \times \stackrel{n}{\cdots} \times X \to X \times \stackrel{n}{\cdots} \times X$,
it can be shown that $\kappa$ is an equivalence of $E_\infty$-tensor structures.  We do not pursue this here.
\end{remark}

\section{Microlocal theory of polyhedral sheaves}
\label{sec:microlocal}

There are three staple operations of microlocal sheaf theory that we will use a lot in the next sections.  They are specialization, microlocalization, and the Fourier-Sato transform.  The general theory of these operations is intricate, and in many interesting cases there is still no satisfactory method of computing them.  But the microlocal theory of \emph{polyhedral} sheaves is simpler.  In this section we give an exposition of microlocal sheaf theory in this simple setting.  Our exposition is not self-contained, but we hope that it will give the reader unfamiliar with microlocal sheaf theory a feel for the basics.

Fix a real vector space $M_\bR$.  We will call a sheaf on $M_\bR$ \emph{polyhedral} if it is constructible with respect to a piecewise-linear stratification of $M_\bR$.  Write $\Sh_{c,\pol}(M_\bR) \subset \Sh_c(M_\bR)$ for the full subcategory of polyhedral sheaves on $M_\bR$.  In this paper almost all constructible sheaves that appear are polyhedral.

We say that a polyhedral sheaf $F$ has \emph{finite type} if it is constructible with respect to a piecewise linear stratification of $M_\bR$ with finitely many strata.  (Note that necessarily some of these strata will be unbounded.  In general we do not require finite type sheaves to have compact support.)  Polyhedral sheaves of finite type form a full triangulated subcategory of $\Sh_{c,\pol}(M_\bR)$, and the class of standard sheaves on finite type polyhedra (i.e. sets cut out by a finite number of linear equations and inequalities) generates this category.  The category is also generated by the class of costandard sheaves on the interiors of finite type polyhedra.

\subsection{Specialization}

For any manifold $M$ and any subset $Y \subset M$, there is a functor $\nu_Y$, called \emph{specialization}, from sheaves on $M$ to sheaves on the normal cone to $Y$ in $M$.  The functor $\nu_Y$ takes values in the category of \emph{conical} sheaves, that is, the category $\Sh_{c,\bR_{>0}}(T_Y M)$ of sheaves that are equivariant for the scaling action of the multiplicative group $\bR_{>0}$ on $T_Y M$.  As $\bR_{>0}$ is contractible, $\Sh_{c,\bR_{>0}}$ can be identified with the full subcategory of $\Sh_c$ whose objects are sheaves that are constant on the orbits of $\bR_{>0}$.  

We are especially interested in $\nu_x(F)$ when $F$ is polyhedral and $x$ is a point of $M_\bR$.  In that case we shall refer to $\nu_x$ as specialization to the tangent space rather than normal cone. The following proposition can function as a definition of $\nu_x$ for the reader unfamiliar with this notion.

\begin{proposition}
\label{prop:nu}
The functor
$$\nu_x:\Sh_{c,\pol}(M_\bR) \to \Sh_{c,\bR_{>0}}(T_x M_\bR)$$
is the unique functor with the following property: for every polyhedral sheaf $F$, there is an open neighborhood $U$ of $x$ such that $\nu_x(F) \vert_U \cong F \vert_U$ under the identification of $U$ with an open neighborhood of $0 \in T_x M_\bR$.
\end{proposition}

Since $\nu_x F$ encodes the first-order behavior of $F$ at $x$, the proposition is a way of saying that the first order behavior of a polyhedral $F$ determines $F$ in a neighborhood of $Y$.  Of course the proposition is false for general constructible sheaves, e.g. for $\nu_0$ of the constant sheaf on the cusp $\{(x,y) \in \bR^2 \mid y^2 = x^3\}$.

\subsection{Fourier-Sato transform}

If $V$ is a vector space and $F$ is a conical sheaf on $V$, there is a sheaf $\FT(F)$ on $V^*$ whose stalks $\FT(F)_\xi$ fit into an exact triangle
$$\FT(F)_\xi \to \Gamma(V;F) \to \Gamma(\{v \in V \mid \xi(v) < -1\};F\vert_{\{\xi(v) < -1\}}) \to$$
(This formula is valid when $\xi \neq 0$.  When $\xi = 0$ we should replace the third term by $\Gamma(V - \{0\};F\vert_{V - \{0\}}$.)
$\FT = \FT_V$ is a functor $\Sh_{\bR_{>0}}(V) \to \Sh_{\bR_{>0}}(V^*)$ called the \emph{Fourier-Sato} transform.  (In the language of \cite{KS}, we have $\FT(F) = F^\wedge$.)  %More generally there is a relative version of $\FT_V$ when $V$ is a vector bundle over a manifold $Y$.

%For $V$ a real $n$-dimensional vector space, or more generally a rank $n$l vector bundle over a manifold $M$, the Fourier-Sato transform is an exact functor $\FT = \FT_V:\Sh_{\bR_{>0}}(V) \to \Sh_{\bR_{>0}}(V^*)$ with $\FT_{V^*} \circ \FT_V(F) \cong -F[-n]$, or more canonically with $\FT_{V^*} \circ \FT_V(F) \cong -F \otimes \ori_V[-n]$.

The following proposition determines the behavior of $\FT$ on polyhedral sheaves, in particular it shows that $\FT$ carries polyhedral sheaves to polyhedral sheaves.

\begin{proposition}
\label{prop:FT}
Let $V$ and $V^*$ be dual vector spaces, and let $\sigma \subset V$ and $\tau \subset V^*$ be dual convex cones.  Let $\sigma^\circ$ and $\tau^\circ$ denote the interiors, and let $s:\sigma^\circ \hookrightarrow V$ and $t:\tau^\circ\hookrightarrow V^*$ be the inclusion maps.  Then we have natural isomorphisms
$$
\begin{array}{c}
\FT_V(s_! \ori_{\sigma^\circ}[\dim(\sigma)]) \cong -t_* \sfR \\
\FT_V(s_* \sfR) \cong t_! \sfR
\end{array}
$$
\end{proposition}

\begin{proof}
This is a special case of Lemma 3.7.10 in \cite{KS}.
\end{proof}

\subsection{Microlocalization, singular support, and the microlocal Morse lemma}

\begin{definition}
Let $F \in \Sh_{c,\pol}(M)$, $x \in M$, and $\xi \in T^*_x M$.
\begin{itemize}
\item The \emph{microlocalization} of $F$ at $x$ is given by
$$\mu_x(F) = \FT(\nu_x(F))$$
Thus, $\mu_x$ is a functor from $\Sh_c(M) \to \Sh_{c,\bR_{>0}}(T^*_x M)$ that takes polyhedral sheaves to conical polyhedral sheaves.
\item The \emph{microlocal stalk} of $F$ at $(x,\xi)$ is the stalk of $\mu_x F$ at $\xi$.  We denote it by $\mu_{x,\xi} F$.
\item The \emph{singular support} of $F$ is the closure of the set of all $(x,\xi)$ such that $\mu_{x,\xi} F \neq 0$.  We denote this set by $\SS(F)$.
\end{itemize}
\end{definition}

Singular support gives us a language to discuss Morse theory for sheaves.  For instance, we say that a point $x \in M_\bR$ is a \emph{critical point} of a sheaf $F$ with respect to a function $f$ if $(x,df_x) \in M_\bR \times N_\bR$ belongs to $\SS(F)$.  We also have the following important analogue of the usual Morse lemma:

\begin{theorem}[Microlocal Morse lemma]
Let $V$ be a real vector space and let $G$ be a polyhedral sheaf on $V$ of finite type.  Let $\eta:V \to \bR$ be a linear function.  Fix $a \in \bR$, and suppose that for all $v \in V$ with $\eta(v) \geq a$, the cotangent vector $(v,d\eta_v)$ does not lie in the singular support of $G$.  Then the restriction map
$$\Gamma(V;G) \to \Gamma(\eta^{-1}(-\infty,a);G)$$
is a quasi-isomorphism.
\end{theorem}

\begin{proof}
The conclusion is the same as that of the usual microlocal Morse lemma \cite[Corollary 5.4.19 (i)]{KS}, except that the hypotheses there require $\eta$ be proper.  The theorem can be proved by reducing it to loc. cit., by compactifying the fibers of $\eta$, or by reducing it to the case where $G$ is a standard sheaf on a finite type polyhedron in $V$, and showing that the restriction map is a quasi-isomorphism directly.  \end{proof}

We have already mentioned that there is a specialization functor $\nu_Y$ for more general subsets $Y \subset M_\bR$.  Similarly, when $Y$ is smooth, there is a functor $\mu_Y$ that takes values in sheaves on the conormal bundle to $Y$, defined as $\FT \circ \nu_Y$ for a relative version of $\FT$.  We do not actually need these more general microlocal operations, but we will make use of the following basic consequence:

\begin{proposition}
\label{prop:microlocalsystem}
Let $F$ be a constructible sheaf and let $\SS(F)^\mathrm{reg}$ be the set of smooth points of $\SS(F)$.  Then there is a locally constant sheaf on $\SS(F)^\mathrm{reg}$ whose stalk at $(x,\xi) \in \SS(F)^\mathrm{reg}$ is naturally isomorphic to $\mu_{x,\xi}(F)$.  In particular, $\mu_x(F)$ is locally constant on $\SS(F)^\mathrm{reg} \cap T_x^* V$.
\end{proposition}

\section{Shard sheaves}
\label{sec:shardsheaves}

Let us define a larger variant of the posets $\bGamma(\Sigma,M)$ defined in Section \ref{sec:three}.  For each real vector space $M_\bR$, we let $\bGamma(M_\bR)$ denote the set of pairs $(\sigma,c)$ where $\sigma$ is a polyhedral cone in the dual space $N_\bR$ and $c$ is a coset in $M_\bR/\sigma^\perp$.  We give this set a partial order by setting $(\sigma,c) \leq (\tau,d)$
whenever $c + \sigma^\vee \subset d + \tau^\vee$.  For each $(\sigma,c) \in \bGamma(M_\bR)$ we define a costandard constructible sheaf $\Theta(\sigma,c)$ on $M_\bR$ by the same formula as Definition \ref{def:theta}:
$$\Theta(\sigma,c) = j_!\omega_{(c + \sigma^\vee)^\circ}$$
where $j$ denotes the inclusion map  $(c+ \sigma^\vee)^\circ \hookrightarrow M_\bR$.

\subsection{Lagrangian shard arrangements}
Let $M_\bR$ be a finite-dimensional real vector space and $N_\bR$ the dual vector space.
We make the following definitions:
\begin{itemize}
\item For each $(\sigma,c) \in \bGamma(M_\bR)$, the \emph{Lagrangian shard} $Z(\sigma,c)$ is the  subset of  $M_\bR \times N_\bR$ given by
$$Z(\sigma,c) = c+\sigma^\perp \times -\sigma = \{(m,n) \mid -n \in \sigma \text{ and } \langle m-x,n'\rangle = 0 \text{ for all }x \in c,\,n' \in \sigma\}$$
\item A subset $\Lambda \subset M_\bR \times N_\bR$ is a \emph{finite shard arrangement} if we have
$$\Lambda = \bigcup_{i= 1}^n Z_i$$
where $Z_1,\ldots,Z_n$ is a finite list of shards.  If $\cZ \subset \bGamma(M_\bR)$ and each $Z_i$ belongs to $\cZ$, then we will say that $\Lambda$ has type $\cZ$.  (Thus, if $\cZ_1 \subset \cZ_2$ and $\Lambda$ has type $\cZ_1$, it also has type $\cZ_2$.)
\end{itemize}

The \emph{height} of a shard $Z(\sigma,c)$ is the dimension of $\sigma$, and the height of a finite shard arrangement is the largest of the heights of its shards.

\subsection{Shard sheaves}
\label{sec:fivetwo}

A \emph{finite shard sheaf} on $M_\bR$ is a constructible sheaf $F$ whose singular support belongs to a finite shard arrangement.  We emphasize that the singular support does not itself have to be a finite shard arrangement.  A finite shard sheaf has type $\cZ$ if its singular support belongs to a finite shard arrangement of type $\cZ$.  The height of a finite shard sheaf is the height of the smallest shard arrangement containing its singular support.

We denote the triangulated dg category of finite shard sheaves by $\Shard(M_\bR) \subset \Sh(M_\bR)$, and we denote the full subcategory of finite shard sheaves of type $\cZ$ by $\Shard(M_\bR;\cZ)$.  If there is a subset $\Lambda \subset T^*M_\bR$ such that $\cZ$ is the collection of all $(\sigma,\chi)$ with $Z(\sigma,\chi) \subset \Lambda$, then we will sometimes write $\Shard(M_\bR;\Lambda) := \Shard(M_\bR;\cZ)$.

\begin{proposition}
\label{prop:thetass}
For each $(\sigma,c) \in \bGamma(M_\bR)$, the costandard sheaf $\Theta(\sigma,c)$
is a finite shard sheaf.
\end{proposition}

\begin{proof}
If $\tau$ is a face of $\sigma$ then for every $y \in \tau^\circ$, the specialization $\nu_y(\Theta(\sigma,x+\sigma^\perp) )$ coincides with the costandard sheaf on $\tau^\vee \subset M_\bR$ under the identification of $M_\bR$ with $T_y M_\bR$.  Thus $\mu_y(\Theta(\sigma,x))$ is the standard sheaf on $-\tau$ by Proposition \ref{prop:FT}.  In particular $\mu_{y,\xi}$ vanishes if $(y,\xi)$ is not in $Z(\tau,x+\tau^\perp)$ for some $\tau$.  It follows that
$$\SS(\Theta(\sigma,x+\sigma^\perp)) \subset \bigcup_\tau Z(\tau,x + \tau^\perp)$$
\end{proof}

The rest of this section is devoted to the proof of the following theorem.

\begin{theorem}
\label{thm:thetagen}
Suppose $\cZ \subset \bGamma(M_\bR)$ satisfies the following conditions:
\begin{enumerate}
\item[(Z1)]  The set of cones $\sigma \subset N_\bR$ such that $(\sigma,c)$ appears in $\cZ$ for some $c$ is a finite polyhedral fan.  (We do not require this fan to be rational.)
\item[(Z2)] If $(\sigma,x + \sigma^\perp) \in \cZ$ and $\tau$ is a face of $\sigma$, then $(\tau,x+\tau^\perp) \in \cZ$.
\end{enumerate}
Then the category $\Shard(M_\bR;\cZ)$ is generated by the sheaves $\{\Theta(\sigma,c) \mid (\sigma,c) \in \cZ\}$.
\end{theorem}

\subsection{The sheaf of homomorphisms into $\Theta$}

An essential ingredient of the proof of Theorem \ref{thm:thetagen} is an application of the microlocal Morse lemma (MML) to the sheaf of homomorphisms $\uhom(F,\Theta(\sigma,c))$.  The hypotheses of the MML require us to understand the singular support of this sheaf.  As a first step we compute its stalks.

\begin{definition}
Let $F$ be a finite shard sheaf on $M_\bR$ of height $h$, and let $\sigma$ be an $h$-dimensional cone in $N_\bR$.  Then $\sigma$ is \emph{narrow} relative to $F$ if for any point $x \in M_\bR$ we either have $\{x\} \times -\sigma \subset \SS(F)$ or $\{x\} \times -\sigma^\circ \cap \SS(F) = \varnothing$.\end{definition}

For instance, if $F$ has type $\cZ$ where $\cZ$ satisfies condition (Z1), and $(\sigma,c) \in \cZ$, then $\sigma$ is narrow with respect to $F$.  The main consequence of narrowness is that, by Proposition \ref{prop:microlocalsystem}, when $\sigma$ is narrow with respect to $F$ the microlocalization $\mu_x(F)$ is constant along the interior of $-\sigma$.  The next lemma expresses this consequence as a property of the stalks of $\uhom(F,\Theta(\sigma,c))$:

\begin{lemma}
\label{lem:shhom}
Let $F$ be a finite shard sheaf on $M_\bR$ of height $h$, and let $\sigma$ be an $h$-dimensional cone that is narrow with respect to $F$.  Let $c$ be a coset of $\sigma^\perp$.  Then for every $x \in c$ and every $\eta \in -\sigma^\circ$, the natural map
$$
\uhom(F,\Theta(\sigma,c))_x \to \dghom(\mu_{x,\eta}(F),\mu_{x,\eta}(\Theta(\sigma,c))) \cong \mu_{x,\eta}(F)^*[-\dim(M_\bR)]
$$
is a quasi-isomorphism.
\end{lemma}

In the lemma we are using the quasi-isomorphism $\mu_{x,\eta} \Theta(\sigma,c) \cong \sfR[-\dim(M_\bR)]$, from Proposition \ref{prop:FT}.

\begin{proof}
In general, if $A$ and $B$ are constructible sheaves then the stalk of $\uhom(A,B)$ at $x$ coincides with $\dghom(A\vert_U,B\vert_U)$ for a sufficiently small neighborhood $U$ of $x$.  By Proposition \ref{prop:nu} it follows that the natural map
$$\uhom(F,\Theta(\sigma,c))_x \to \dghom(\nu_x F,\nu_x \Theta(\sigma,c))$$
is a quasi-isomorphism.  Since $\FT$ is an equivalence of categories, the map
$$\uhom(F,\Theta(\sigma,c))_x \to \dghom(\mu_x F,\mu_x \Theta(\sigma,c))$$
is also a quasi-isomorphism.  If $i$ denotes the inclusion of $-\sigma^\circ$ into $N_\bR$, then we have $\mu_x \Theta(\sigma,c) = i_* \sfR[-\dim(M_\bR)]$, so by adjunction the map
$$\uhom(F,\Theta(\sigma,c))_x \to \dghom(i^* \mu_x F,\sfR[-\dim(M_\bR)])$$
is a quasi-isomorphism.

Since $\sigma$ is narrow with respect to $F$, $\mu_x(F)$ is constant on the interior of $-\sigma$, so $\dghom(i^* \mu_x F, \sfR[-\dim(M_\bR)])$ is naturally quasi-isomorphic to $\dghom(\mu_{x,\eta} F,\sfR[-\dim(M_\bR)])$ for every point $\eta \in -\sigma^\circ$.  This completes the proof.
\end{proof}

\begin{remarks}
\begin{enumerate}
\item
Let $F$ be a height $h$ shard sheaf and $\sigma$ an $h$-dimensional cone narrow with respect to $F$.  Then we have the following consequence of the lemma: in order for $\SS(F)$ to contain $\{x\} \times -\sigma$, it must contain the whole shard $Z(\sigma,x+\sigma^\perp)$.

\item The Lemma and Proposition \ref{prop:microlocalsystem} together imply that the stalks of $\uhom(F,\Theta(\sigma,c))$ are the stalks of a local system on $c$.  (In fact it can be shown that under the hypotheses of the Lemma, $\uhom(F,\Theta(\sigma,c))$ itself is locally constant along $c$.)

\end{enumerate}
\end{remarks}

\subsection{D\'evissage lemma and the proof of Theorem \ref{thm:thetagen}}
\label{sec:dev}

Let $F$ be a finite shard sheaf of height $h$, and let $\sigma$ be a cone that is narrow with respect to $F$.  We say that $(\sigma,c)$ is \emph{blocked} with respect to $F$ if there is some homomorphism from $F$ to $\Theta(\sigma,c)$ defined in a neighborhood of $c$ that does not extend to a global homomorphism.  More precisely, we say that $(\sigma,c)$ is blocked if
$$\dghom(F,\Theta(\sigma,c)) \to \uhom(F,\Theta(\sigma,c))_x$$
fails to be a quasi-isomorphism for some (and hence every) $x \in c$.

To prove the d\'evissage lemma we need the following easy (but somewhat confusing) fact:

\begin{proposition}
\label{prop:dev}
Let $\sigma, \tau, \upsilon \subset N_\bR$ be strictly convex polyhedral cones.  Suppose that $\tau$ and $\upsilon \cap \sigma$ are proper faces of $\sigma$.  Form the intersection $C = (\tau^\vee)^\circ \cap (-\upsilon^\vee)^\circ$.  Then either $(-C^\vee) \cap \sigma^\circ$ is empty or $C$ is empty.
\end{proposition}

\begin{proof}
Note that the following are equivalent:
\begin{itemize}
\item $C$ is nonempty.
\item There is a linear function $f:N_\bR \to \bR$ that takes positive values on $\tau - \{0\}$ and negative values on $\upsilon - \{0\}$.
\item  $\tau \cap \upsilon = 0$.
\end{itemize}

Denote the convex hull of $\tau$ and $-\upsilon$ by $H$. We will
show that any of the conditions above implies $C^\vee=H$. A point
$\xi$ in $H$ is of the form $a\xi_t-b\xi_u$ for $a,b\ge0$, and
$\xi_t\in \tau$ and $\xi_u\in\upsilon$. For any linear function
$g\in C$, since $g(\xi_t)\ge0$ and $g(\xi_u)\le 0$,
$g(\xi)=g(a\xi_t-b\xi_u)\ge 0$, and thus $\xi\in C^\vee$, i.e.
$H\subset C^\vee$. On the other hand, let $f$ be a linear function
given by the second condition above. The convex hull $H$ is
characterized by a collection of linear functions $f_i(\xi)\ge 0$,
$i=1,\dots,l$. For any point $\xi\not\in H$, there exists an $i_0$
such that $f_{i_0}(\xi)<0$. Choose a small positive $\epsilon$ such
that $(f_{i_0}+\epsilon f)(\xi)<0$. Notice that $f_{i_0}$ is
non-negative on $\tau$ and non-positive on $\upsilon$, thus
$f_{i_0}+\epsilon f$ is positive on $\tau-\{0\}$ and negative on
$\upsilon-\{0\}$, and then $f_{i_0}+\epsilon f\in C$. Therefore
$\xi\not \in C^\vee$, and the other direction $C^\vee \subset H$
follows.

Suppose $C$ is non-empty. If $\xi$ belongs to $-C^\vee$, then $\xi$ is of the form $\xi_u - \xi_t$, where $\xi_u \in \upsilon$ and $\xi_t \in \tau$.  But then if $\xi \in \sigma$, we have $\xi_u = \xi_t + \xi \in \sigma$.  Thus $(-C^\vee) \cap \sigma \subset \upsilon \cap \sigma$.  Since we have assumed $\upsilon \cap \sigma$ is a proper face of $\sigma$, this completes the proof.
\end{proof}

The d\'evissage lemma says the following:

\begin{lemma}
\label{lem:dev}
Let $\cZ \subset \bGamma(M_\bR)$ satisfy conditions (Z1) and (Z2) of Theorem \ref{thm:thetagen}, and let $F$ be a finite shard sheaf of type $\cZ$ and height $h$.  Then there is a $(\sigma,c) \in \cZ$ such that $\sigma$ is $h$-dimensional, $\dghom(F,\Theta(\sigma,c)) \neq 0$, and $(\sigma,c)$ is not blocked for $F$.
\end{lemma}

\begin{proof}
Let $(\sigma,c)$ be such that $\sigma$ is $h$-dimensional, $Z(\sigma,c) \subset \SS(F)$, and any other $d$ with $Z(\sigma,d) \subset \SS(F)$ has $(\sigma,c) < (\sigma,d)$ in the partial order on $\bGamma(M_\bR)$.  Thus $c$ are the points in $M_\bR$ ``farthest away in the $\sigma^\vee$ directions'' for which $F$ contains $\sigma$ in its singular support.  By Lemma \ref{lem:shhom}, $\uhom(F,\Theta(\sigma,c))$ does not vanishes along $c$.  We will show that $(\sigma,c)$ is not blocked for $F$, so that $\dghom(F,\Theta(\sigma,c)) = \uhom(F,\Theta(\sigma,c))_x \neq 0$ for each $x \in c$.

By the microlocal Morse lemma, to show that $(\sigma,c)$ is not blocked for $F$ it suffices to show that
$(c + \sigma^\vee - c) \times \sigma^\circ \cap \SS(\uhom(F,\Theta(\sigma,c))) = \varnothing$.  Equivalently, it suffices to show that for each $x \in c + \sigma^\vee$ that is not in $c$, there is a neighborhood $V$ of $x$ with $V \times \sigma^\circ \cap \SS(\uhom(F,\Theta(\sigma,c))\vert_V) = \varnothing$.  By Proposition \ref{prop:nu} this is equivalent to showing that $T_x M_\bR \times \sigma \cap \SS(\uhom(\nu_x F, \nu_x\Theta(\sigma,c))) = \varnothing$.

Let us identify $T_x M_\bR$ with $M_\bR$.  Since $x \notin c$, there is a proper face $\tau \subset \sigma$ such that $\nu_x \Theta(\sigma,c) \cong \Theta(\tau,0)$.  Now $\nu_x F$ is the inverse Fourier transform of a sheaf on $N_\bR$ that is constructible with respect to the cones in $-\Sigma$ of dimension $\leq h$.  Thus to show that $\SS(\uhom(\nu_x F,\nu_x \Theta(\sigma,c))$ does not meet $M_\bR \times \sigma^\circ$, it suffices to show that $\SS(\uhom(G,\nu_x \Theta(\sigma,c)))$ does not meet $T_x M_\bR \times \sigma^\circ$ whenever $\FT(G)$ is a costandard sheaf on a cone $-\upsilon \in -\Sigma$ of dimension $\leq h$.

By Proposition \ref{prop:FT} such a $G$ has $G = i_* \sfR$, where $i:-\upsilon^\vee \hookrightarrow M_\bR$ is the inclusion of $-\upsilon^\vee$ into $M_\bR$.  We compute $\uhom(G,\Theta(\tau,0)) = j_! \sfR$ (up to a shift), where $j$ is the inclusion $j:((-\upsilon)^\vee \cap \tau^\vee)^\circ \hookrightarrow M_\bR$.  Proposition \ref{prop:dev} shows that $\sigma^\circ$ is not in the singular support of $j_! \sfR$ if $(-\upsilon^\vee)^\circ\cap(\tau^\vee)^\circ$ is non-empty. On the other hand when this is empty $j_! \sfR$ is zero.  This completes the proof.
\end{proof}

\begin{proof}[Proof of Theorem \ref{thm:thetagen}]
Let $\ltr_{\cZ} \subset \Sh_c(M_\bR)$ be the full triangulated category generated by
$$\{\Theta(\sigma,c) \mid (\sigma,c) \in \cZ\}$$
Since $\cZ$ contains the flat shard $M_\bR \times \{0\}$, and any sheaf of height $0$ is constant, the theorem holds when $F$ has height $0$.  Suppose now that $h>0$.  We will prove the following claim: if $F$ is of type $\cZ$ and height $\leq h$, we can find another sheaf $F'$ and a map $F' \to F$ with the following properties:
\begin{itemize}
\item $F'$ has height $< h$;
\item the cone on $F' \to F$ is generated by sheaves of the form $\Theta(\sigma,c)$, where each $\sigma$ is $h$-dimensional and each $(\sigma,c)$ belongs to $\cZ$.
\end{itemize}
The theorem follows from the claim by induction.

Let $F$ be a shard sheaf of type $\cZ$ and height $\leq h$.  Let us say that $F$ has $h$-complexity $\leq n$ if the singular support of $F$ is contained in a union of shards, at most $n$ of which have height $h$.  We will prove the claim by induction on the $h$-complexity of $F$.  If $F$ has $h$-complexity $\leq 0$, then $F$ has height $< h$ and so belongs to $\ltr_\cZ$.  Suppose that the $F$ has $h$-complexity $\leq n$ and that we have proven that all sheaves of $h$-complexity $\leq {n-1}$ belong to $\ltr_\cZ$.

By Lemma \ref{lem:dev}, there is a $\sigma$ of dimension $h$ and a $c \in M_\bR/\sigma^\perp$ such that $(\sigma,c)$ is not blocked for $F$ and such that $\dghom(F,\Theta(\sigma,c)) \neq 0$.  Consider the exact triangle
$$F \to \dghom(F,\Theta(\sigma,c))^* \otimes \Theta(\sigma,c) \to F'' \to $$
where the first map is the coevaluation map of $F$.  Then $F$ belongs to $\ltr_{\cZ}$ if and only if $F''$ does.  Since $(\sigma,c)$ is unblocked, we can conclude by Lemma \ref{lem:shhom} that the first map induces an isomorphism on $\mu_{x,\eta}$ whenever $x \in c$ and $\eta \in -\sigma^\circ$.  It follows that $Z(\sigma,c)$ is not in the singular support of $F''$, so that $F''$ has $h$-complexity $<n$.  This completes the proof of the claim and the theorem.
\end{proof}

\subsection{An open problem on locally finite shard sheaves}

Let $\Sigma \subset N_\bR$ be a rational polyhedral fan.  Then
\begin{equation}
\LS = \bigcup_{\tau\in\Sigma} (\tau^\perp + M)\times -\tau \subset M_\bR\times N_\bR = T^*M_\bR
\end{equation}
is a \emph{locally} finite shard arrangement.  Let $\cZ \subset \bGamma(M_\bR)$ denote the set of all $(\sigma,\chi)$ such that $Z(\sigma,\chi) \subset \LS$.  We have a nested triple of categories
$$\Sh_{c}(M_\bR;\LS) \supset  \Shard(M_\bR,\cZ) \supset \Sh_{cc}(M_\bR,\LS)$$
i.e. the category of constructible sheaves with singular support in the locally finite shard arrangement $\LS$, the category of such constructible sheaves whose singular support belongs to a finite subarrangment of $\LS$, and the category of such sheaves with compact support.

In this paper, we relate the smaller two categories to categories of quasicoherent sheaves on the toric variety $X$ corresponding to $\Sigma$:
\begin{enumerate}
\item We prove in Section \ref{sec:perfect} that, when $X$ is proper, $\Sh_{cc}(M_\bR;\LS)$ is quasi-equivalent to the dg category of perfect complexes of equivariant coherent sheaves on $X$.
\item We prove in Section \ref{sec:finfib} that $\Shard(M_\bR;\cZ)$ is quasi-equivalent to the dg category of quasicoherent sheaves on $X$ with finite fibers.
\end{enumerate}
On the other hand we do not understand how the largest category $\Sh_c(M_\bR;\LS)$ is related to quasicoherent sheaves on $X$.  A better understanding is important if we wish to obtain analogues of the results of this paper for nonequivariant sheaves.  The conical Lagrangian $\LS$ is stable under lattice translations, and defines a conical Lagrangian $\tLS \subset T^*(M_\bR/M)$ in the cotangent bundle of the compact torus $M_\bR/ M$.  The techniques of Section \ref{sec:three}  can be extended (see \cite{Tr}) to prove that, for proper toric varieties $X$, $\Perf(X)$ admits a full embedding into $\Sh_{c}(M_\bR/M; \tLS) = \Sh_{cc}(M_\bR/M;\tLS)$.  It is natural to expect that every constructible sheaf in $\Sh_c(M_\bR/M;\tLS)$ corresponds to a perfect complex on $\Perf(X)$, but we are not able to prove this.

It is almost but not quite possible to deduce the nonequivariant results from equivariant results by some formal arguments.  One would like to study an object of $\Sh_c(M_\bR/M;\tLS)$ by pulling it back to an object of $\Sh_c(M_\bR;\LS)$ (with an $M$-equivariant structure), but this pullback object necessarily cannot have compact support, nor is it a finite shard sheaf.  It is therefore not at all clear that such a thing can be understood in terms of $\Theta$-sheaves.  It would be helpful (but not definitive) to settle the following problem, which we state as a conjecture:
  
\begin{conjecture}
Let $F$ be a \emph{locally finite} shard sheaf of type $\cZ$, i.e. suppose that $F$ is constructible and that the singular support of $F$ belongs to a locally finite union of shards $Z(\sigma,c)$ with $(\sigma,c) \in \cZ$.  If $\dghom(F,\Theta(\sigma,c)) = 0$ for all $(\sigma,c) \in \cZ$, then $F = 0$.
\end{conjecture}

After Theorem \ref{thm:thetagen}, the conjecture is equivalent to the statement that any locally finite shard sheaf can be written as a direct limit of finite shard sheaves.

\section{Quasicoherent sheaves with finite fibers}

\label{sec:finfib}
Recall that in the theory of coherent (and quasicoherent) sheaves we make a distinction between the \emph{stalk} and the \emph{fiber} of a sheaf of $\cO$-modules at a point.  The former is a module over a local ring $\cO_x$---in particular is not usually finitely generated over $\sfR$---and the latter is a vector space over the residue field $k(x) = \cO_x/\mathfrak{m}_x$.  The fibers of a coherent sheaf are always finite dimensional (but possibly unbounded cohomologically, if we take the left-derived functor of the fiber functor), but the converse is far from true.  For instance, the fibers of the quasicoherent sheaf $\Theta'(\sigma,\chi)$ are all either one- or zero-dimensional.

\begin{definition}
A complex of quasicoherent sheaves $\cF$ on an $\sfR$-scheme $X$ has \emph{finite fibers} if for each $\sfR$-valued point $x:\Spec \sfR \to X$ of $X$, the image $x^*\cF$ of $\cF$ under the pullback
$$\cQ(X) \to \cQ(\Spec \sfR)$$
is perfect.  If $X$ carries a group action, then we say that an equivariant quasicoherent sheaf has finite fibers if the underlying non-equivariant sheaf does.
\end{definition}

If $\sfR$ is an algebraically closed field, then a map $x:\Spec \sfR \to X$ is given by a closed point in $X$.  The definition is equivalent to requiring that for all closed points $x$,
\begin{itemize}
\item $\Tor_i(\cO_x/\mathfrak{m}_x,\cF) := h^{-i}(\cO_x/\mathfrak{m}_x \stackrel{\mathbf{L}}{\otimes} \cF)$ are finite-dimensional, and
\item $\Tor_i(\cO_x/\mathfrak{m}_x,\cF) = 0$ for all but finitely many $i \in \bZ$.
\end{itemize}
If $X$ has singularities, the condition is not satisfied by all coherent sheaves.  However it is satisfied by all vector bundles and perfect complexes.

\begin{remark}
\label{rem:finfib}
It follows from the adjunction formula
$$\dghom_{\sfR}(x^* \cF,\sfR) \cong \dghom(\cF,x_* \sfR)$$
that $\cF$ has finite fibers if and only if $\dghom(\cF,x_* \sfR)$ is a perfect $\sfR$-module for each $\sfR$-valued point $x$.

\end{remark}

The sheaves $\Theta'(\sigma,\chi)$ on a toric variety $X$ have finite fibers.  Our goal is to prove that they generate the full triangulated subcategory $\cQ_T^{\finfib}(X) \subset \cQ_T(X)$ of equivariant quasicoherent sheaves with finite fibers.  We are grateful to Bhargav Bhatt for suggesting this result to us.

\begin{theorem}
\label{thm:finfib}
For any toric variety $X$, there is a quasi-equivalence $\ltrp \cong \cQ_T^\finfib(X)$.
\end{theorem}

This will be proved in Section \ref{sec:ssqc}.  The theorem suggests a natural question: that the category of quasicoherent sheaves with finite fibers on any scheme is generated by localizations of the structure sheaf.  The question can be reduced to the affine case, where our question is more precisely as follows:

\begin{question}
Let $\cO$ be a commutative ring, let $D^b(\cO)$ denote the bounded derived category of $\cO$-modules, and let $D^{b,\finfib}(\cO)$ denote the full subcategory of complexes with the following finiteness properties:
\begin{itemize}
\item
$\Tor_i(M,\cO_\mathfrak{p}/\mathfrak{p})$ is a finite-dimensional $\cO_{\mathfrak{p}}/\mathfrak{p}$-module for every prime ideal $\mathfrak{p} \subset \cO$.
\item $\Tor_i(M,\cO_\mathfrak{p}/\mathfrak{p}) = 0$ for all but finitely many $i$.
\end{itemize}
Is $D^{b,\finfib}(\cO)$ generated as a triangulated category by localizations, i.e. by modules of the form $\cO[S^{-1}]$, where $S$ is an arbitrary multiplicative system in $\cO$?
\end{question}

Our proof of Theorem \ref{thm:finfib} does not shed much light on this question.  It is modeled on the proof of Theorem \ref{thm:thetagen}.

\subsection{``Singular support'' of equivariant quasicoherent sheaves}
\label{sec:ssqc}

In this section we prove Theorem \ref{thm:finfib}.  Since Theorem \ref{thm:finfib} (and Theorem \ref{thm:mainfirst}) is true, there are analogs of the constructible notions of ``singular support'' and of ``height'' for quasicoherent sheaves with finite fibers.  We prove Theorem \ref{thm:finfib} by finding these notions (or serviceable replacements for them) a priori, and following the proof of Theorem \ref{thm:thetagen}.

\begin{definition}
Let $X$ be a toric variety with corresponding fan $\Sigma$.  For each integer $h$, let $X^{(h)} \subset$ denote the open toric subvariety of $X$ obtained by removing all $T$-orbits of codimension greater than $h$, and let $j$ denote the inclusion map $X^{(h)} \hookrightarrow X$.  We say that a quasicoherent sheaf $\cF \in \cQ_T(X)$ has \emph{height $\leq h$} if the following equivalent conditions are satisfied:
\begin{enumerate}
\item $\cF \to j_* j^* \cF$ is a quasi-isomorphism.
\item $i^* \cF = 0$ whenever $i$ is the inclusion of a $T$-orbit of codimension $> h$.
\end{enumerate}
\end{definition}

(A proof that conditions (1) and (2) are equivalent in general can be found in \cite[Lemma 10]{arinkin}.)

For each $(\sigma,\chi) \in \bGamma(\Sigma,M)$, define
$$\theta'(\sigma,\chi):= i_* \cO_{O_\sigma}(\chi)$$
where $O_\sigma$ is the orbit corresponding to $\sigma$, $i$ is the inclusion map $i:O_\sigma \hookrightarrow X$, and $\cO_{O_\sigma}(\chi)$ is the structure sheaf of $O_\sigma$ with equivariant structure given by $\chi$.  From the $i^*$--$i_*$-adjunction, we see that $\cF$ has height $< h$ if and only if $\dghom(\cF,\theta'(\sigma,\chi)) = 0$.

Let $\cF \in \cQ_T(X)$ have finite fibers, of height $\leq h$.  Let $\SS_h(\cF)\subset \bGamma(\Sigma,M)$ be the set of all $(\sigma,\chi)$ with $\dim(\sigma) = h$ and $\dghom(\cF,\theta'(\sigma,\chi))\neq 0$.   Let us say that $(\sigma,\chi) \in \SS_h(\cF)$ is \emph{unblocked} for $\cF$ if the map
$$\dghom(\cF,\Theta'(\sigma,\chi)) \to \dghom(\cF,\theta'(\sigma,\chi))$$
is a quasi-isomorphism.

\begin{remark}
Our terminology is justified by (but logically independent of) the following observations.
\begin{enumerate}
\item If we assume that Theorem \ref{thm:finfib} is true, then it is easy to see that $\kappa$ applied to a quasicoherent sheaf of height $h$ yields a constructible sheaf of height $h$.
\item Moreover, if $x$ is a point on a codimension $h$ orbit, and $\cF$ has height $h$, then the weight spaces of $\cF_x$ (under the action of the isotropy subgroup $T_x \subset T$) are naturally isomorphic to microlocal stalks of $\kappa(\cF)$.
\item As a consequence of (2), the union of all shards $Z(\sigma,\chi)$ for $(\sigma,\chi) \in \SS_h(\cF)$ is the height $h$ part of $\SS(\kappa(\cF))$.
\end{enumerate}
\end{remark}

\begin{lemma}
\label{lem:devdas}
Let $\cF \in \cQ_T(X)$ have finite fibers and height $h$.  If $\SS_h(\cF) \neq \varnothing$, then there is a $(\sigma,\chi) \in \SS_h(\cF)$ that is not blocked for $\cF$.
\end{lemma}

\begin{proof}
Since $\cF$ has finite fibers, $\SS_h(\cF)$ is a finite set.  Thus, $\SS_h(\cF)$ has a minimal element $(\sigma,\chi)$ with respect to the partial order on $\bGamma(\Sigma,M)$.  We will show that this element is not blocked for $\cF$.  Without loss of generality we may assume that $\chi$ is the coset of $0 \in M_\bR$.  We have to show that
$$\dghom(\cF,\Theta'(\sigma,0)) \to \dghom(\cF,\theta'(\sigma,0))$$
is a quasi-isomorphism.  By the adjunction between $j_{\sigma*}$ and $j_\sigma^*$, this is equivalent to showing that
$$\dghom(\cF\vert_{X_\sigma},\cO_{X_\sigma}) \to \dghom(\cF\vert_{X_\sigma},\cO_{O_\sigma})$$
is a quasi-isomorphism.  If $\cI$ denotes the ideal sheaf $O_\sigma$ in $X_\sigma$, with its natural $T$-equivariant structure, then this is equivalent to showing that $\dghom(\cF\vert_{X_\sigma},\cI) = 0$.  By choosing a total order of $\sigma^\vee \cap M$ that refines the usual partial order, we can endow $\cI$ with a filtration by equivariant submodules whose subquotients are of the form $\cO_{O_\sigma}(\xi_n)$, with each $\xi_n \in \sigma^\vee$ and $\xi_n \neq 0$.
Thus we have a spectral sequence
$$h^i(\dghom(\cF\vert_{X_\sigma},\cO_{O_\sigma}(\xi_j))\to h^{i+j}(\dghom(\cF\vert_{X_\sigma},\cI))$$

But the hypothesis that $(\sigma,0)$ is a minimal element of $\SS_h(\cF)$ shows that the $E_2$ term of this spectral sequence vanishes.  This completes the proof.
\end{proof}

\begin{proof}[Proof of Theorem \ref{thm:finfib}]
Since any sheaf of height 0 is of the form $\Theta'(0,\chi)$, $\ltrp$ contains all quasicoherent sheaves with finite fibers of height zero.  Suppose now that $h > 0$.  We will prove the following claim: if $\cF$ has finite fibers and is of height $\leq h$ we can find another quasicoherent sheaf $\cF'$ and a map $\cF' \to \cF$ with the following properties:
\begin{itemize}
\item $\cF'$ has height $< h$;
\item the cone on $\cF' \to \cF$ is generated by sheaves of the form $\Theta'(\sigma,\chi)$, where each $\sigma$ is $h$-dimensional.
\end{itemize}
The theorem follows from the claim by induction.

Suppose $\cF$ has finite fibers and height $\leq h$.  We will prove the claim by induction on the size of $\SS_h(\cF)$.  If $\SS_h(\cF)$ is empty, then $\cF$ has height $< h$ and so belongs to $\ltrp$.  Suppose now that $\SS_h(\cF)$ has $n$ elements and that we have proven that all sheaves for which $\SS_h$ have $< n$ elements belong to $\ltrp$.

By Lemma \ref{lem:devdas}, there is a $\sigma$ of dimension $h$ and a $(\sigma,\chi) \in \bGamma(\Sigma,M)$ with $\dghom(F,\Theta'(\sigma,\chi)) \neq 0$ and $(\sigma,\chi)$ is not blocked for $\cF$.  Consider the exact triangle
$$\cF \to \dghom(\cF,\Theta'(\sigma,\chi))^* \otimes \Theta'(\sigma,\chi) \to \cF'' \to $$
where the first map is the coevaluation map of $\cF$.  Then $\cF$ belongs to $\ltrp$ if and only if $\cF''$ does.  Applying $\dghom(-,\theta'(\tau,\xi))$ to the triangle gives
$$\dghom(\cF'',\Theta'(\tau,\xi)) \to \dghom(\cF,\Theta'(\tau,\xi)) \otimes \dghom(\cF,\Theta'(\sigma,\chi)) \to \dghom(\cF,\theta'(\tau,\xi)) \to $$
We see that $\dghom(\cF'',\Theta'(\tau,\xi))$ vanishes whenever $\dghom(\cF,\Theta'(\tau,\xi))$ does, and also when $(\tau,\xi) = (\sigma,\chi)$.  It follows that $\SS_h(\cF'')$ has fewer elements than $\SS_h(\cF)$ and so $\cF''$ belongs to $\ltrp$ by the induction hypothesis.  This completes the proof.
\end{proof}

As a corollary of Theorem \ref{thm:finfib}, Theorem \ref{thm:thetagen}, and Theorem \ref{thm:mainfirst}, we have the following theorem valid for an arbitrary toric variety.

\begin{theorem}
\label{thm:mainsecond}
Let $X$ be a toric variety corresponding to a fan $\Sigma$.  Then there is a quasi-equivalence of dg categories
$$\kappa:\cQ_T^{\finfib}(X) \stackrel{\sim}{\to} \Shard(M_\bR;\LS).$$
\end{theorem}

%\subsection{Coherent-constructible dictionary---fibers and microlocal stalks}

%In this section we show that, under the equivalence $\cQ_T^\finfin(X) \cong \Shard(M_\bR;\LS)$ of theorem \ref{thm:mainsecond}, the notions of height and (in some sense) of singular support used in the proofs of theorem \ref{thm:finfib} and \ref{thm:thetagen} match:

%\begin{theorem}
%Let $X$ be a toric variety corresponding to a fan $\Sigma$.  If $\cF \in \cQ_T(X)$ has height $h$ in the sense of this section, then $\kappa(\cF)$ has height $h$ in the sense of section \ref{sec:fivetwo}.  Moreover, if $x \in X$ belongs to an orbit corresponding to an $h$-dimensional cone $\sigma$, and $\chi$ is a character of the isotropy subgroup $T_x \subset T$, then there is a natural isomorphism
%$$\cF_{x,\chi} \cong \mu_{y,\xi} \kappa(\cF)$$
%where
%\begin{itemize}
%\item $\cF_{x,\chi}$ denotes the $\chi$-weight space of the action of $T_x$ on the fiber $\cF_x$.

%\item $\mu_{y,\xi}$ denotes the microlocal stalk at a point $y \in \chi + \sigma^\perp \subset M_\bR$ and covector $\xi \in -\sigma^\circ$
%\end{itemize}
%\end{theorem}

%\begin{proof}

%\end{proof}

\section{Perfect complexes and compactly supported sheaves}
\label{sec:perfect}

In this section we characterize the image of perfect complexes under the coherent-constructible correspondence.

\begin{theorem}
\label{thm:perfect}
Let $X$ be a proper toric variety corresponding to a fan $\Sigma \subset N_\bR$.  Let $\kappa:\ltrp \to \ltr$ be the functor constructed in Section \ref{sec:three}.
\begin{enumerate}
\item If $\cE \in \ltrp$ is perfect, then $\kappa(\cE) \in \ltr$ has compact support.
\item The resulting functor $\Perf_T(X) \to \Sh_{cc}(M_\bR;\LS)$ is a quasi-equivalence.
\end{enumerate}

\end{theorem}

%\begin{remark}
%There exist complete toric varieties (necessarily neither smooth nor projective) for which it is \emph{unknown} whether $\Perf_T(X_\Sigma)$ contains any nontrivial objects at all \cite{Payne}.  For such toric varieties it appears to be just as difficult to construct objects of $\Sh_{cc}(M_\bR,\Lambda_\Sigma)$.
%\end{remark}

\begin{proof}
We prove (1) here, and postpone the proof of (2) to Section \ref{sec:dualizable}.

By passing to a resolution of singularities, it suffices (after Example \ref{ex:blowup}) to assume that $X$ is smooth.  Similarly by Chow's lemma for toric varieties, we may furthermore assume that $X$ is projective.  An argument of Seidel's (see \cite[Proposition 1.3]{A2}) shows that every equivariant vector bundle on a smooth projective toric variety has a bounded resolution by line bundles, so it suffices to show that $\kappa(\cL)$ has compact support when $\cL$ is a line bundle.   It is not difficult to prove this directly, but we will deduce it by applying Theorem \ref{thm:linbun}.  Since $X$ is smooth and projective and $\kappa$ is monoidal,
$$\begin{array}{c}
\cL \cong \cL_1 \otimes \cL_2^{-1}, \\
\kappa(\cL) \cong \kappa(\cL_1) \star \kappa(\cL_2^{-1}),
\end{array}
$$
where $\cL_1$ and $\cL_2$ are ample.  That $\kappa(\cL_1)$ has compact support follows from Theorem \ref{thm:linbun} and that $\kappa(\cL_2^{-1}) = -\cD(\kappa(\cL_2))$ has compact support follows from Theorem \ref{thm:linbun} and from Theorem \ref{thm:dual} below.
\end{proof}

The rest of the proof of Theorem \ref{thm:perfect} relies on the following characterization of perfect complexes.

\begin{proposition}
\label{prop:dualizable}
Let $\cE \in \cQ_T(X)$.  The following are equivalent:
\begin{enumerate}
\item $\cE$ is perfect---i.e. $\cE$ is locally quasi-isomorphic to a bounded complex of $T$-equivariant vector bundles.
\item $\cE$ is strongly dualizable---i.e. there exists an object $\cF \in \cQ_T(X)$, together with maps $e:\cE \otimes \cF \to \cO$ and $c:\cO \to \cF \otimes \cE$, such that the composite map
$$\cE \to \cE \otimes \cF \otimes \cE \to \cE$$
is the identity.
\end{enumerate}
\end{proposition}

See \cite{BFN} for a discussion of the history of this result, and a proof in a modern context.

\subsection{Coherent-constructible dictionary---duality}
\label{sec:dualizable}

The category $\Perf_T(X)$ has a natural contravariant involution: the internal hom $\cE \mapsto \uhom(\cE,\cO)$.  (Note that when $X$ is not smooth, or at least not Gorenstein, this is different from the Grothendieck duality of \cite{Ha}---the structure sheaf is not a dualizing sheaf in that case).  Write $\cE \mapsto \cE^\vee$ for this duality.  (The object $\cE^\vee$ is canonically isomorphic to any object $\cF$ in Proposition \ref{prop:dualizable}(2)).  Recall that $\cD:\Sh_c(M_\bR)^o \to \Sh_c(M_\bR)$ denotes the Verdier duality functor.  This functor changes singular support by the map $M_\bR \times N_\bR \to M_\bR \times N_\bR: (x,y) \mapsto (x,-y)$.  It follows that $\cD$ composed with the map $-1:M_\bR \to M_\bR$ preserves the subcategory $\Sh_c(M_\bR;\LS)$.

\begin{theorem}
\label{thm:dual}
Suppose $X$ is a complete toric variety, and let $\cE \in \Perf_T(X)$.  There is a natural quasi-isomorphism
$$\kappa(\cE^\vee) \cong -\cD(\kappa(\cE)).$$
\end{theorem}

\begin{proof}
As indicated in proposition \ref{prop:dualizable}, the functor $\cE \mapsto \cE^\vee$ is characterized by the tensor product.  More precisely, it is the unique functor for which there exists a natural quasi-isomorphism
$$
\dghom(\cE_1 \otimes \cE_2,\cO) \cong \dghom(\cE_1,\cE_2^\vee).
$$
So by corollary \ref{cor:tensor}, to prove the theorem it is enough to show that we have a similar natural isomorphism
$$
\dghom(F_1 \star F_2,\delta) \cong \dghom(F_1,-\cD(F_2))
$$
where $\delta = \kappa(\cO)$ denotes the skyscraper sheaf at $0 \in M_\bR$ (placed in cohomological degree 0).  Let $v$ denote the addition map $v:M_\bR \times M_\bR \to M_\bR$.  By adjunction we have
$$
\dghom(F_1 \star F_2,\delta) = \dghom(v_!(F_1 \boxtimes F_2),\delta) \cong \dghom(F_1 \boxtimes F_2,v^!\delta).
$$
Let $Z = \{(x,y) \in M_\bR \times M_\bR \mid x+ y = 0\}$, and let $i$ denote the inclusion map $i:Z \hookrightarrow M_\bR \times M_\bR$.  One computes $v^! \delta \cong i_* \ori_Z[\dim(Z)] = i_* \omega_Z$.  Under the map
$$M_\bR \times M_\bR \to M_\bR \times M_\bR:(x,y) \mapsto (x,-y)$$
the set $Z$ is carried to the diagonal, so that we have
$$\begin{array}{cl}
\dghom(F_1 \star F_2,\delta) & \cong \dghom(F_1 \boxtimes (-F_2),\Delta_* \omega_{M_\bR})\\
& \cong \dghom(F_1 \otimes (-F_2),\omega_{M_\bR}) \\
& \cong \dghom(F_1,\cD(-F_2)).
\end{array}
$$
This completes the proof.
\end{proof}

The functor $-\cD:\Sh_c(M_\bR)^o \to \Sh_c(M_\bR)$ is well-defined, but not every object in $\Sh_c(M_\bR)$ is strongly dualizable with respect to the convolution product.  The problem is that the map
$$F \to F \star (-\cD F) \star F \to F$$
can fail to be the identity map.  (This is the case even when $F$ is the constant sheaf on $M_\bR$.)  However we have the following:

\begin{lemma}
Let $F \in \Sh_c(M_\bR)$.  Suppose that $F$ is polyhedral and has compact support.  Then $F$ is strongly dualizable with respect to the convolution product.
\end{lemma}

\begin{proof}
We may reduce to the case where $F$ is a standard sheaf on a simplex.  In that case $-\cD F$ is a costandard sheaf on the antipodal simplex.  We claim that $F \star -\cD F$ is quasi-isomorphic to $\delta$, the skyscraper sheaf at 0.  Then we can take the maps $c$ and $e$ of Proposition \ref{prop:dualizable}(2) to be inverse quasi-isomorphisms.  We have the following formula for the stalk of $F \star -\cD F$ at a point $x \in M_\bR$:
$$
(F \star -\cD F)_x = \Gamma(\{(y,z) \in M_\bR \times M_\bR \mid y + z = x\};F \boxtimes -\cD F).
$$
If $F$ is the standard sheaf on a simplex $\Delta \subset M_\bR$, then the cohomology groups of the right-hand side are isomorphic to the relative cohomology groups
$$
H^{i+n}(\Delta \cap (x + \Delta),\Delta \cap (x + \partial\Delta)).
$$
This group vanishes unless $x = 0$, in which case it is the reduced cohomology of the $n$-sphere $\Delta/\partial\Delta$.  This completes the proof.
\end{proof}

\begin{proof}[Proof of Theorem \ref{thm:perfect}]
We have shown that $\kappa:\Perf_T(X_\Si) \to \Sh_{cc}(M_\bR;\LS)$ is fully faithful.  It remains to show that it is essentially surjective.  Suppose $F$ is a sheaf on $M_\bR$ with singular support in $\LS$.  Suppose furthermore that $F$ has compact support.  Then $F$ is a shard sheaf, so by Theorem \ref{thm:thetagen} and Theorem \ref{thm:mainfirst} there is a quasicoherent sheaf $\cG \in \ltrp$ with $\kappa(\cG) \cong F$.  We need to show that $\cG$ is perfect, or equivalently by Proposition \ref{prop:dualizable} that $\cG$ is strongly dualizable.  By the lemma, $\kappa(\cG) = F$ is strongly dualizable, and moreover the dual $-\cD F$ is a shard sheaf.  It follows that there is an $\cH$ such that $\kappa(\cH) = -\cD F$ and therefore $\cG$ is strongly dualizable.  This completes the proof.
\end{proof}

\end{document}